%% file: ergodic-spdes-12.tex
\documentclass[12pt,reqno]{amsart} 

\usepackage[T1]{fontenc}
\usepackage{enumitem}

\usepackage{amsmath, amsfonts, amsthm, color, tikz, comment, xcolor, xfrac, amssymb}
\usepackage{amsrefs}

\renewcommand{\MR}[1]{~\href{https://mathscinet.ams.org/mathscinet-getitem?mr=MR#1}{MR#1}.}

\BibSpec{article}{%
    +{}  {\PrintAuthors}                {author}
    +{,} { \textit}                     {title}
    +{.} { }                            {part}
    +{:} { \textit}                     {subtitle}
    +{,} { \PrintContributions}         {contribution}
    +{.} { \PrintPartials}              {partial}
    +{,} { }                            {journal}
    +{}  { \textbf}                     {volume}
    +{}  { \PrintDatePV}                {date}
    +{,} { \issuetext}                  {number}
    +{,} { \eprintpages}                {pages}
    +{,} { }                            {status}
    +{,} { \url}                        {url}
    +{,} { \PrintDOI}                   {doi}
    +{,} { available at \eprint}        {eprint}
    +{}  { \parenthesize}               {language}
    +{}  { \PrintTranslation}           {translation}
    +{;} { \PrintReprint}               {reprint}
    +{.} { }                            {note}
    +{.} {}                             {transition}
    +{}  {\SentenceSpace \PrintReviews} {review}
}

\usepackage[normalem]{ulem}
\usepackage{mathrsfs}
\usepackage{mathtools}

\usepackage{pdfsync}
\usepackage[font={scriptsize}]{caption}
\usepackage{environ}

\usepackage[colorlinks=true,
            linkcolor=blue,
            filecolor=magenta,
            urlcolor=blue,
            citecolor=magenta,
            pdftitle={Ergodic SPDEs},
            pdfpagemode=FullScreen,
            breaklinks=true,
           ]{hyperref}

\usepackage[left=1in, right=1in, top=1.1in,bottom=1.1in]{geometry}
\numberwithin{equation}{section}  

\newcommand{\im}{\mbox{Im}}

\newcommand{\iot}{\int_{0}^{t}}

\newcommand{\ou}{[0,1]}

\newcommand{\ud}{\ensuremath{\mathrm{d}}}
\newcommand{\Norm}[1]{\left\|  #1   \right\|}

\newcommand{\cj}{\mathcal{J}}

\newcommand{\cn}{\mathcal{N}}

\newcommand{\EE}{\mathbb{E}}
\newcommand{\E}{\mathbb{E}}

\newcommand{\R}{\mathbb{R}}

\newcommand{\al}{\alpha}

\newcommand{\vp}{\varphi}

\newtheorem{theorem}{Theorem}[section]

\newtheorem{definition}[theorem]{Definition}

\newtheorem{lemma}[theorem]{Lemma}
\newtheorem{notation}[theorem]{Notation}

\newtheorem{proposition}[theorem]{Proposition}

\theoremstyle{remark}
\newtheorem{remark}[theorem]{Remark}

\theoremstyle{remark}
\newtheorem{example}[theorem]{Example}

\begin{document}

\title[Ergodicity for stochastic PDEs]{On ergodic properties of stochastic PDEs}

\author[L. Chen]{Le Chen}
\address{L. Chen: Department of Mathematics and Statistics, Auburn University, Auburn}
\email{le.chen@auburn.edu}

\author[C. Ouyang]{Cheng Ouyang}
\address{C. Ouyang: Department of Mathematics, Statistics, and Computer Science, University of Illinois at Chicago, Chicago}
\email{couyang@uic.edu}

\author[S. Tindel]{Samy Tindel}
\address{S. Tindel: Department of Mathematics, Purdue University, West Lafayette}
\email{stindel@purdue.edu}

\author[P. Xia]{Panqiu Xia}
\address{P. Xia: School of Mathematics, Cardiff University, Cardiff, United Kingdom}
\email{xiap@cardiff.ac.uk}

\maketitle
\vspace{-2em}

\begin{center}
  \today
\end{center}

\begin{center}
 \emph{Dedicated to Giuseppe (Beppe) Da Prato, in memoriam}
\end{center}

\begin{abstract}
  In this note we review several situations in which stochastic PDEs exhibit
  ergodic properties. We begin with the basic dissipative conditions, as stated
  by Da Prato and Zabczyk in their classical monograph. Then we describe the
  singular case of SPDEs with reflection. Next we move to some degenerate (and
  thus more demanding) settings. Namely we recall some results obtained around
  2006, concerning stochastic Navier-Stokes equations with a very degenerate
  noise. We finish the article by handling some cases with degenerate
  coefficients. This includes a new result about the parabolic Anderson model in
  dimension $d\ge 3$, driven by a general class of noises and fairly general
  initial conditions. In this context, a phase transition is observed, expressed
  in terms of the noise intensity. \medskip

   \medskip

  \noindent\textit{MSC class.~} 60H15; 35R60; 37L40; 35K05.

\end{abstract}


\begingroup
\hypersetup{linkcolor=black}
\tableofcontents
\endgroup

\section{Introduction}

Any person having attended a stochastic analysis conference in the 90's-00's
will certainly remember Beppe, consistently beginning his talks with the famous
singing words: ``let $A$ be an operator on a Hilbert space $H$, generating a
$C_{0}$-semigroup $S(t)$\ldots''. However, behind what could sometimes be seen
as a kind of ritual, lies an important fact: Beppe was an exceptional pioneer in
the area of stochastic PDEs, who was completely passionate about his topic. Most
importantly, he certainly was the undisputed leader of the field for decades.
This outstanding legacy is well highlighted by the series of
books~\cite{da-prato.zabczyk:96:ergodicity, da-prato.zabczyk:02:second,
da-prato.zabczyk:14:stochastic} written in collaboration with Jerzy Zabczyk,
which have been and still are an immense source of inspiration for the
stochastic PDE community.

Among the influential contributions mentioned above, the current paper will
single out the monograph~\cite{da-prato.zabczyk:96:ergodicity}. This book brings
together the worlds of general ergodic theory and stochastic models in infinite
dimensions, a delicate task which is achieved by Da Prato and Zabczyk in an
astonishingly clear and consistent way. Starting from the foundations, the book
covers a wide range of applications to systems (dissipative equations,
Navier-Stokes equations, spin systems) which are still the object of active
research today.

Our contribution first proposes to review a few basic facts contained
in~\cite{da-prato.zabczyk:96:ergodicity}. Namely we will recall the fundamental
notions allowing to get ergodic type results for infinite dimensional stochastic
differential equations seen as dynamical systems. We will illustrate this
technology by stating an ergodic result for a generic stochastic PDE. We shall
then examine some cases departing from the standard Da Prato-Zabczyk setting.
Without any pretension to be exhaustive, we wish to give an account on the
following situations:

\begin{enumerate}[wide, labelwidth=!, labelindent=0pt, label=\textbf{(\roman*)}]
  \setlength\itemsep{.1in}

  \item \emph{Case of reflected equations.}
    Following~\cite{zambotti:01:reflected}, we will see how reflections change
    the landscape in terms of invariant measures. This encompasses the
    identification of new invariant measures, as well as integration by parts
    formulae.

  \item \emph{Ergodicity in degenerate situations.} In the celebrated
    result~\cite{hairer.mattingly:06:ergodicity}, Hairer and Mattingly showed
    that a nondegenerate noise is not necessary for good mixing properties of a
    stochastic PDE. Their result focuses on a stochastic Navier-Stokes equation,
    for which only four active modes in the noise are requested.

  \item \emph{Phase transition in ergodicity.} Our own contribution will be
    related to another degenerate situation, namely the parabolic Anderson model
    (PAM). There we will see that for a spatial variable $x\in\R^{d}$ with $d\ge
    3$, we have two cases: at high temperature the equation exhibits an ergodic
    behavior; In contrast, for low temperatures the moments of the solution blow
    up as $t\to\infty$. We shall give a self contained proof of the ergodic
    behavior based on contraction arguments, for a general spatial covariance of
    the noise. This result appears to be new and is inspired
    by~\cite{gerolla.hairer.ea:23:fluctuations, gu.li:20:fluctuations}.

\end{enumerate}

\noindent This constant dialogue with Da Prato and Zabczyk's books (seen as
benchmarks) is a staple in the stochastic analysis literature. We will also
mention possible new developments for polymer measures.

\medskip

Our article is structured as follows: in Section~\ref{S:Standard} we recall the
classical Da Prato-Zabczyk ergodic setting for stochastic PDEs.
Section~\ref{S:ReflectSHE} is devoted to equations with reflection. In
Section~\ref{S:NS} we turn to the case of degenerate noises, focusing on the
stochastic Navier-Stokes case. Eventually Section~\ref{S:PAM} deals with phase
transitions in ergodicity for generalized PAMs. A Gronwall-type lemma is
proved in Appendix~\ref{S:Gronwall}.


\section{The standard stochastic PDE case}\label{S:Standard}

In this section, we describe the standard ergodic setting for stochastic PDEs in
Da Prato and Zabczyk's language. We will spell out the general setting in
Section~\ref{SS:Abstract} and then outline the main ergodic result in
Section~\ref{SS:Ergodic}.

\subsection{Abstract setting for stochastic PDEs}\label{SS:Abstract}

Let us briefly recall the abstract setting which is used
in~\cite{da-prato.zabczyk:96:ergodicity} in order to solve stochastic
differential equations in infinite dimensions. It can be split in four
ingredients (notice that for sake of conciseness we will not define every
technical term below, we refer to Beppe's original book for a complete and
self-contained exposition):

\begin{enumerate}[wide, labelwidth=!, labelindent=0pt, label=\textbf{(\alph*)}]
  \setlength\itemsep{.1in}

  \item\label{it:dpz-setting-a} As recalled in the introduction, everything
    starts with the infinitesimal generator $A$ of a strongly continuous
    semigroup $\{S(t); t \ge 0\}$ defined on a separable Hilbert space $H$
    (notice that analyticity of $S(t)$ is further assumed
    in~\cite{da-prato.zabczyk:96:ergodicity}).

  \item\label{it:dpz-setting-b} One considers a mapping $F:H\to H$ with linear
    growth and Lipschitz continuity. That is one has
    \begin{align}\label{E:StdF}
      F(x)\lesssim 1+ |x|, \quad
      \mathrm{and}\quad |F(x)-F(y)|\lesssim |x-y|,
    \end{align}
    where $|\cdot|$ denotes the norm in $H$. The mapping $F$ will be the drift
    of our equation.

  \item\label{it:dpz-setting-c} On another (possibly larger) Hilbert space $U$,
    let $Q$ be a trace class operator. Then we can define a $Q$-Wiener process
    on a complete probability space $(\Omega, \mathcal{F}, \mathbb{P})$. This is
    a $U$-valued continuous process $\{ W_{t}; t \ge 0\}$ with independent
    increments and such that for all $t \ge 0, \, h > 0$ we have
    \begin{align}\label{E:Noise-STD}
      \frac{1}{h^{1/2}}\left(W_{t+h}-W_t\right) \sim \cn(0, Q).
    \end{align}
    Notice that the space $U$ is equipped with an inner product induced by $Q$:
    \begin{align}\label{E:InnerProdU}
      \langle \xi, \psi\rangle_Q \coloneqq \langle\xi, Q\psi\rangle_U.
    \end{align}

  \item\label{it:dpz-setting-d} The diffusion coefficient $B$ of our equation
    has to be compatible with the $Q$-Wiener process and the semigroup $S(t)$.
    Namely $B: H\to L(U;H)$ has to be a strongly continuous map such that the
    following growth conditions are fulfilled for every $t>0$ and $x,y\in H$:
    \begin{gather}
      \|S(t)B(x)\|_{\mathrm{HS}}\le K(t) (1+|x|), \\
      \|S(t)B(x)-S(t)B(y)\|_{\mathrm{HS}}\le K(t)|x-y|,
    \end{gather}
    where $K$ is a kernel which sits in the space $L^2([0,T])$ for every $T>0$.
    We also label the following assumption which is useful in order to get
    regularity properties of our stochastic differential equations: there exists
    $\alpha\in (0,1/2)$ such that for every $T>0$ we have
    \begin{align}\label{E:Cond_K}
      \int_0^T s^{-2\alpha}K^2(s)\ud s<\infty.
    \end{align}
    As we will see in the next section, the above setting is sufficient to
    ensure existence, uniqueness and ergodic properties of a general class of
    stochastic PDEs.

\end{enumerate}

\subsection{Ergodic properties}\label{SS:Ergodic}

With our setting of Section~\ref{SS:Abstract} in hand, the general type of
$H$-valued equation considered in~\cite{da-prato.zabczyk:96:ergodicity} can be
written as
\begin{align}\label{E:general}
  \ud X_t=\left[ AX_t+F(X_t)\right]\ud t+B(X_t)\ud W_t,
\end{align}
with an initial condition $X_0=\xi$ for a given $\xi\in H$. Observe that
equation~\eqref{E:general} is solved in the so-called mild sense, which can be
written as
\begin{align}\label{E:G-mild}
  X_t=S(t)\xi+\int_0^tS(t-s)F(X_s)\ud s+\int_0^tS(t-s)B(X_s)\ud W_s,
\end{align}
where the stochastic integral in~\eqref{E:G-mild} has to be interpreted in the
It\^{o} sense.

The conditions~\ref{it:dpz-setting-a}--\ref{it:dpz-setting-d} in
Section~\ref{SS:Abstract} ensure existence and uniqueness of the solution
to~\eqref{E:general}, although a complete theory for this fact is better found
in the first Da Prato-Zabczyk volume~\cite{da-prato.zabczyk:02:second}. As far
as ergodic properties are concerned, it should be noticed that ergodic behaviors
often rely on damping terms in dynamic systems. In case of
equation~\eqref{E:general}, this damping is provided by the operator $A$ (whose
spectrum is implicitly thought of as mostly negative). The main result in this
direction is summarized in the following theorem.
\begin{theorem}\label{T:MomCond}
  Assume the setting of Section~\ref{SS:Abstract} holds true. In addition,
  suppose that the operators $S(t)$ are compact and that the following uniform
  bound on moments is satisfied: one can find $T_0>0$ and $p\ge 2$ such that
  \begin{align}\label{E:MomCond}
    \sup_{t\ge T_0, \, \xi\in H}\mathbb{E}_\xi\left[ |X_t|^p\right]<\infty.
  \end{align}
  Then there exists an invariant measure for equation~\eqref{E:general}.
\end{theorem}
\begin{proof}
  The detailed proof is found in \cite[Theorem
  6.1.2]{da-prato.zabczyk:96:ergodicity}. We will content ourselves with an
  outline of the main ingredients. This will be split in several steps.

  \noindent \textit{Step 1:~Averaged measure.}\quad The solution $\{X_t;
  t\ge0\}$ to equation~\eqref{E:general} generates a Markov dynamics. We call
  $P_t(\xi,\cdot)$ the corresponding transition. Next for a Borel set $\Gamma\in
  H$ and $T>0$ we set
  \begin{align}\label{E:ave}
    R_T(\xi,\Gamma)=\frac{1}{T}\int_0^TP_t(\xi,\Gamma)\ud t.
  \end{align}
  According to a celebrated criteria from Krylov and Bogoliubov, there exists an
  invariant measure for $X_t$ as long as the family $\{R_{T}; T>0\}$ is tight.
  Let us recall that tightness of $R_{T}$ means that for every
  $\epsilon\in(0,1)$, one is able to find a compact set $K_\epsilon\subset H$
  such that
  \begin{align}\label{E:tight}
    R_T(\xi,K_\epsilon)\ge 1-\epsilon,\quad\mathrm{uniformly\ in}\ T.
  \end{align}
  Summarizing, relation~\eqref{E:tight} implies the existence of an invariant
  measure.

  A criterion like~\eqref{E:tight} is easy to verify when $H=\mathbb{R}^n$.
  Indeed, in that case the compact $K_\epsilon$ can be taken as a (closed) ball
  $B(0, r_\epsilon)$ for a proper radius $r_\epsilon>0$. In addition, a direct
  application of Markov's inequality entails that for any $p\ge1$
  \begin{align}\label{E:Markov-Ineq}
    \mathbb{P}\left(X_t\notin B(0,r_\epsilon)\right) \le \frac{\mathbb{E}\left[|X_t|^p\right]}{r^p_\epsilon},
  \end{align}
  so that~\eqref{E:tight} is easily implied by~\eqref{E:MomCond}. This is not
  true anymore when $H$ is an infinite dimensional functional space. Some
  additional ingredients are thus needed.

  \noindent \textit{Step 2:~Factorization method.} In order to solve the problem
  raised in Step 1, Da Prato and Zabczyk resort to one of their most emblematic
  and clever trick. Namely under assumption~\ref{it:dpz-setting-d}, some
  elementary fractional integral arguments show that the stochastic integral
  in~\eqref{E:G-mild} can be expressed as
  \begin{align}\label{E:Factor}
    \int_0^tS(t-s)B(X_s)\ud W_s=\frac{\sin(\alpha\pi)}{\pi}\int_0^t(t-s)^{\alpha-1}S(t-s)Y_s\ud s,
  \end{align}
  where the process $Y=Y^{(\alpha)}$ is the element of $L^2([0,T]; H)$ defined
  by
  \begin{align*}
    Y_s=\int_0^s(s-r)^{-\alpha}S(s-r)B(X_r)\ud W_r.
  \end{align*}
  Formula~\eqref{E:Factor} is a progress in the following sense: we have
  obtained a representation of stochastic convolutions as Lebesgue type
  integrals. This expression is more amenable to computations for regularity
  (and therefore compactness) estimates.

  \noindent \textit{Step 3:~Compactness.}\quad Recall that our goal is to
  establish relations~\eqref{E:tight}-\eqref{E:Markov-Ineq} for a family of
  compact sets $K_\epsilon$. This is achieved by writing the solution
  to~\eqref{E:general} at time $t=1$ in the following way (thanks to the
  factorization formula~\eqref{E:Factor}):
  \begin{align}\label{E:X1}
    X_1=S(1)\xi+G_1F(X)+G_\alpha Y,
  \end{align}
  where for $\alpha\in(0,1]$ we introduce an operator $G_\alpha: L^p([0,1];
  H)\to H$ by
  \begin{align*}
    G_\alpha f=\frac{\sin(\alpha\pi)}{\pi}\int_0^1(1-s)^{\alpha-1}S(1-s)f(s)\ud s.
  \end{align*}
  An important point in the proof is that whenever $S(t)$ is compact for $t>0$,
  $G_\alpha$ is also compact. Therefore whenever $\xi, F(X)$ and $Y$
  in~\eqref{E:X1} lye in proper balls in $H$, the random element $X_1$ belongs
  to a compact set. A full description of this family of sets $\{K_\epsilon;
  r>0\}$ stems from the decomposition~\eqref{E:X1}:
  \begin{align}\label{E:Def-Ke}
    \begin{aligned}
      K_\epsilon=\bigg\{  x\in H;\quad
      & x=S(1)\xi+G_1g+G_\alpha h\, ,\quad \text{with} \\
      & |\xi|\le \frac{1}{\epsilon},\quad
        |g|_{L^p([0,1]; H)}\le\frac{1}{\epsilon},\quad\text{and}\quad
        |h|_{L^p([0,1];H)}\le \frac{1}{\epsilon}\bigg\}.
    \end{aligned}
  \end{align}

  \noindent \textit{Step 4:~Conclusion.} The expression~\eqref{E:Def-Ke} allows
  to locate $X_1$ given by~\eqref{E:X1} in a compact set thanks to mere moment
  estimates. Furthermore, we have assumed in~\eqref{E:MomCond} that the moments
  of $X_t$ are uniformly bounded. Hence for $\epsilon>0$ and $K_\epsilon$
  defined by~\eqref{E:Def-Ke}, we end up with
  \begin{align*}
    \mathbb{P}_\xi(X_1\in K_\epsilon)\ge 1-C\epsilon^p(1+|\xi|^p).
  \end{align*}
  An easy Markov conditioning procedure allows then to conclude
  that~\eqref{E:tight} holds true whenever $\epsilon$ is small enough. This
  implies the existence of an invariant measure for $X_t$ as explained in
  Step~1.
\end{proof}

As the reader might see from the sketch above, Da Prato and Zabczyk's method is
a very elaborate combination of deep functional analysis insight, stochastic
calculus techniques and fundamental ergodic tools. It is difficult to overstate
Da Prato's pioneering role in this type of development. Let us add a couple of
remarks to close the section.

\begin{remark}
  Existence of an invariant measure is the most basic result one can get in
  ergodic theory. The extra ingredients put forward
  in~\cite{da-prato.zabczyk:96:ergodicity} in order to get uniqueness are
  essentially reduced to strong Feller properties for the transition $P_t$ and
  irreducibility. We will go back to those issues in Section~\ref{S:NS}.
  However, let us mention at this point that uniqueness of the invariant measure
  is closely related to irreducibility of the Markovian transition
  $P_{t}(\xi,\cdot)$ alluded to in the proof of Theorem~\ref{T:MomCond}. This
  irreducibility is ensured by nondegeneracy properties of the noise $W$ and the
  diffusion coefficient $B$ in~\eqref{E:G-mild}. A standard way to express this
  nondegeneracy when $B$ is constant is the following:
  \begin{equation}\label{E:a1}
    \im\left( S(t)  \right) \subset \im\big( Q_{t}^{1/2}  \big) \, ,
    \quad\text{where}\quad
    Q_{t}=\iot S(s) B\, S(s)^{*} \, \ud s \, .
  \end{equation}
\end{remark}

\begin{remark}
  For dissipative systems, that is when $A+F$ in~\eqref{E:general} can be
  considered as a damping coefficient, one can prove convergence of
  $\mathscr{L}(X_t)$ to an invariant measure by a direct analysis based on
  It\^{o} formula. This type of consideration is obtained, e.g.,
  in~\cite[Section 6.3 and 6.5]{da-prato.zabczyk:96:ergodicity}
\end{remark}

\begin{remark}
  As mentioned in the introduction, Da Prato and Zabczyk also investigate the
  ergodic behavior of specific stochastic systems of interest. Those include
  delayed systems, Burger's equation, spin systems, and the Navier-Stokes
  equation. Each case requires some serious modification of the general method.
\end{remark}

\section{A stochastic heat equation with reflection}\label{S:ReflectSHE}

We now start a series of deviations from the standard setting for ergodic
stochastic PDEs (as introduced in Section~\ref{S:Standard}), beginning with
equations involving reflections. Observe that stochastic heat equations with
reflection were firstly studied by Nualart and
Pardoux~\cite{nualart.pardoux:92:white} in a quasilinear setting, building upon
a deterministic framework that had previously received considerable attention
from various researchers; see~\cite{bensoussan.lions:78:applications} and
references therein. It is worth noting that reflected stochastic heat equations
can also arise as the scaling limit of microscopic models for random
interfaces~\cite{etheridge.labbe:15:scaling, funaki.olla:01:fluctuations}. As
far as ergodic properties are concerned, reflections induce an extra singularity
in the equation. In the current section, we will summarize an approximation
procedure allowing us to handle this singularity. For further exploration of the
topic, readers are referred to Zambotti's book~\cite{zambotti:17:random} and
survey paper~\cite[Section 6]{zambotti:21:brief}.

Stochastic PDEs with reflections can be handled in the abstract setting of
Section~\ref{S:Standard}. In this context, the underlying Hilbert space should
be $H = L^2 ([0,1])$. Nevertheless, since we wish to encompass Dirichlet
boundary conditions, we will consider a state space of the form
\begin{equation}\label{c1}
  H_0 \coloneqq \left\{\xi \in H ;\; \xi(0) = \xi(1) = 0\right\},
\end{equation}
equipped with the inner product inherited from $H$. In addition, the reflection
measure will constrain solutions to live in a family of sets $\{K_{\alpha}; \,
\alpha \ge 0\}$ defined by
\begin{align}\label{E:def-Kalpha}
  K_{\alpha} \coloneqq C([0,1]) \cap \left\{\xi \in H_0 ;\; \xi \ge - \alpha\right\}\, .
\end{align}
With this preliminary notation in mind, we now consider the following equation
driven by a Brownian sheet $W$:
\begin{align}\label{E:reflect}
  \ud X_t = \left[A X_t + F(X_t)\right] \ud t + \ud W_t +  \eta(\ud t).
\end{align}
In equation~\eqref{E:reflect}, the operator $A$ is defined as $A = \frac{1}{2}
\Delta$, and $F$ is assumed to be a mapping similar to that considered in
\eqref{E:StdF}, which in particular might take the form
\begin{align}\label{E:def-f-ref}
  \left[F(\xi)\right] (x) = f(x, \xi(x)), \quad \text{for all $\xi \in H$ and $x \in [0,1]$},
\end{align}
with a regular enough function $f: [0,1]\times \R \to \R$. As mentioned above,
$W$ is a Brownian sheet. This means, in the language of
Section~\ref{S:Standard}, that $W$ is a $Q$-Wiener process as
in~\eqref{E:Noise-STD} with $Q = \mathrm{Id}_H$. In the literature, the time
derivative of $W$ (interpreted as a generalized function) is also referred to as
space-time white noise.

It has been established in~\cite{nualart.pardoux:92:white} that under certain
mild conditions, for any initial condition $\xi \in K_\alpha$, the
equation~\eqref{E:reflect} reflected at $- \alpha \le 0$ admits a unique
solution. This solution is a pair $(X^{\alpha}, \eta^{\alpha})$ wherein
$X^{\alpha}$ is a $K_{\alpha}$-valued process adapted to $W$ starting at $\xi$,
and $\eta^{\alpha}$ is an adapted nonnegative random measure satisfying
\begin{align*}
  \int_{\R_+} \int_0^1 \big(\alpha + X_t^{\alpha}(x) \big) \, \eta^{\alpha} (\ud t, \ud x) = 0.
\end{align*}

In terms of ergodic results for~\eqref{E:reflect}, it is natural to wonder how
the singularity of the reflection is affecting invariant measures. Let us thus
say a few words about invariant measures for a system analogous
to~\eqref{E:reflect}, yet with no reflection. In other words, let us consider an
equation such as~\eqref{E:general} with $B = \mathrm{Id}$ and $F$ of the
form~\eqref{E:def-f-ref}. For any $\xi \in H_0$ (recall that $H_0$ is defined
by~\eqref{c1}), we set
\begin{align}\label{E:def-Gxi}
  G(\xi) \coloneqq \int_0^1 \ud x \int_0^{\xi (x)} f(x,y) \ud y.
\end{align}
Now we define an underlying measure $\mu$ on $H_0$ as the distribution of the
Brownian bridge. It is a well-known fact (see the proof of
Theorem~\ref{T:inv-reflect} below for more details about this assertion) that an
invariant measure $\mu^{G}$ for~\eqref{E:general} with $B = \mathrm{Id}$ is
given by
\begin{align}\label{E:def-inv-ou-g}
  \mu^G (\ud \xi) \coloneqq e^{2 G(x)} \mu (\ud x).
\end{align}
In his paper~\cite{zambotti:01:reflected}, Zambotti shows how the above
expansion for the invariant measure is perturbed by the reflection. The result
can be summarized as follows:

\begin{theorem}[Theorem~5 of~\cite{zambotti:01:reflected}]\label{T:inv-reflect}
  Let $B_t$ be the standard Brownian motion in $\R^3$. Let $\nu$ be the
  probability measure of the 3-$d$ Bessel bridge, namely, the law of
  $\left\{\left|B_t\right|\right\}_{t \in [0,1]}$ conditioned to $B_1 = 0$.  The
  set $K_{0}$ is introduced in \eqref{E:def-Kalpha} with $\al=0$. We consider a
  measure $\nu^{G}$ on $K_0$ defined as follows: for all $\xi \in K_0$,
  \begin{align}\label{E:inv-reflect}
    \nu^G (\ud \xi) \coloneqq e^{2 G(\xi)} \nu(\ud \xi),
  \end{align}
  where $G$ is given by~\eqref{E:def-Gxi}. Then, $\nu^G$ is an invariant measure
  for~\eqref{E:reflect} reflected at $0$.
\end{theorem}
\begin{proof}
  Just as in Theorem~\ref{T:MomCond}, we only provide a brief overview of the
  proof strategy, directing interested readers to~\cite{zambotti:01:reflected}
  for a more detailed exposition.

  \noindent \textit{Step 1:~Approximating the equation.} Let $\alpha > 0$. For
  any $\epsilon > 0$, consider the following equation
  \begin{align}\label{E:reflect-appox}
    \ud X_t^{\alpha, \epsilon} = \left[ A  X_t^{\alpha, \epsilon} + F(X_t^{\alpha,\epsilon}) + \frac{(\alpha + X_t^{\alpha, \epsilon})^-}{\epsilon} \right] \ud t + \ud W_t,
  \end{align}
  with $r^- \coloneqq \max \{-r , 0\}$ for all $r \in \R$. Then, it is known
  that for any initial condition $X_0 = \xi \in H_0$, the
  equation~\eqref{E:reflect-appox} has a unique solution in $H_0$. Furthermore,
  assuming the initial condition $\xi \in K_{\alpha}$, the same monotonicity
  arguments as in~\cite{nualart.pardoux:92:white} assert that as $\epsilon
  \downarrow 0$, the following statements hold:

  \begin{enumerate}[wide, labelwidth=!, labelindent=0pt, label=\textbf{(\roman*)}]
    \setlength\itemsep{.1in}

    \item For any $(t,x) \in \R_+ \times [0,1]$, $X_t^{\alpha, \epsilon} (x)$
      increases to a random variable $X_t^{\alpha} (x)$ such that $X_t^{\alpha}
      = \left\{X_t^{\alpha} (x) ; x \in [0,1]\right\} \in K_{\alpha}$ for all $t
      > 0$;

    \item $\displaystyle \left\{\eta^{\alpha, \epsilon} (\ud t, \ud x) \coloneqq
      \epsilon^{-1}(\alpha + X_t^{\alpha, \epsilon} (x))^- \ud t \ud
    x\right\}_{\epsilon > 0}$ converge weakly to a nonnegative measure
    $\eta^{\alpha}$ on $\R_+ \times [0,1]$;

    \item The pair $(X^{\alpha}, \eta^{\alpha})$ is a solution
      to~\eqref{E:reflect} reflected at $-\alpha$.

  \end{enumerate}

  \noindent \textit{Step 2:~Invariant measure for the OU process.} Let $A$ be
  any infinitesimal generator of a strongly continuous semigroup on $H_0$ as in
  Section~\ref{SS:Abstract}. Then the \textit{Ornstein-Uhlenbeck process} (OU
  process), which is the solution to
  \begin{align*}
    \ud Z_t = A Z_t \ud t + \ud W_t \, ,
  \end{align*}
  has an invariant measure $\mu$, which is a centered Gaussian measure on $H_0$
  with covariance operator $\int_0^{\infty} e^{2 A t} \ud t = (-2A)^{-1}$. This
  fact had been proved in Da Prato-Zabczyk's monograph \cite[Theorem
  6.2.1]{da-prato.zabczyk:96:ergodicity}. As alluded to
  in~\eqref{E:def-inv-ou-g}, in case $A = \frac{1}{2} \Delta$, the Gaussian
  measure $\mu = \mathcal{N}\left(0, (- \Delta)^{-1}\right)$ on $H_0$,
  concentrated on $H_0 \cap C([0,1])$, coincides with the distribution of the
  1-$d$ Brownian bridge $\mathcal{B}(t) \stackrel{\mathrm(d)}{=} B_t - t B_1$
  for all $t \in [0,1]$ (where $B$ stands for a 1-$d$ Brownian motion). This can
  be seen by the following equality, valid for every pair $h ,k \in H_0$:
  \begin{align}\label{E:Bridge}
    \E \left[\langle \mathcal{B}, h \rangle \langle \mathcal{B}, k \rangle \right]
    = & \int_0^1 \ud t \int_0^1 \ud s \, \E \left[\left(B_s - s B_1\right) \left(B_t - t B_1 \right)\right] h(s) k(t) \nonumber \\
    = & \int_0^1 \ud t \int_0^1 \ud s \left(s \wedge t - st\right) h(s) k (t) = \int_0^1 (-\Delta)^{-1} h(t) k(t) \ud t.
  \end{align}
  Notice that in~\eqref{E:Bridge}, the last inequality can be briefly justified
  as follows: for every $h \in H_0$, the solution to the Dirichlet problem
  \begin{align*}
    \Delta g (t) = h(t), \quad \text{with boundary condition $g(0) = g(1) = 0$},
  \end{align*}
  can be written as
  \begin{align*}
    g(t) = - \int_0^1 \ud s \left(s \wedge t - st\right) h(s).
  \end{align*}

  \noindent \textit{Step 3:~Invariant measure for~\eqref{E:reflect-appox}.} We
  now justify the claim in~\eqref{E:def-inv-ou-g} in our context. Let $G$ be
  given as in~\eqref{E:def-Gxi}. It is easy to see that for any $\xi, \xi' \in
  H_0$, we have
  \begin{align*}
    \langle \nabla G(\xi), \xi' \rangle_{H_0}
    = \int_0^1 f(x, \xi(x)) \, \xi'(x) \, \ud x
    = \langle f(\cdot, \xi(\cdot)), \xi' \rangle_{H_0}.
  \end{align*}
  In the same way, if $V_{\alpha}$ is a mapping on $H_0$ defined by
  \begin{align*}
    V_{\alpha} (\xi) \coloneqq \frac{1}{2} \int_0^1 \left[\left(\alpha + \xi(x) \right)^-\right]^2 \ud x, \quad \text{for all $\xi \in H_0$},
  \end{align*}
  then it is readily checked that for all $\xi, \xi' \in H_0$, the following
  holds
  \begin{align*}
    \langle \nabla V_{\alpha}(\xi), \xi' \rangle_{H_0} = - \left\langle \left(\alpha + \xi(\cdot) \right)^-, \xi' \right\rangle_{H_0}.
  \end{align*}
  As a consequence,~\eqref{E:reflect-appox} can be written as the following
  gradient system:
  \begin{align*}
    \ud X_t^{\alpha, \epsilon} = \left[ A  X_t^{\alpha, \epsilon} + \nabla G(X_t^{\alpha,\epsilon}) - \frac{1}{\epsilon} \nabla V_{\alpha} (X_t^{\alpha, \epsilon}) \right] \ud t + \ud W_t.
  \end{align*}
  Therefore, it follows from~\cite[Section 8.6]{da-prato.zabczyk:96:ergodicity}
  that equation~\eqref{E:reflect-appox} has an invariant measure of the form
  \begin{align}\label{E:iv-reflect-alp-eps}
    \nu_{\alpha,\epsilon}^G (\ud \xi) \coloneqq \frac{1}{Z_{\alpha, \epsilon}} e^{2G(\xi) - (2V_{\alpha} (\xi)/\epsilon) } \mu(\ud \xi),
  \end{align}
  where, under some integrability conditions (not specified here for the sake of
  conciseness), the renormalization constant $Z_{\alpha,\epsilon}$ defined below
  is a finite positive number:
  \begin{align*}
    Z_{\alpha, \epsilon} \coloneqq \int_{H_0}  e^{2G(\xi) - (2V_{\alpha} (\xi)/\epsilon) } \mu (\ud \xi) \in (0,\infty).
  \end{align*}

  \noindent \textit{Step 4:~Convergence of $\nu^G_{\alpha, \epsilon}$.} Our
  claim~\eqref{E:inv-reflect} will be obtained by taking limits in both $\alpha$
  and $\epsilon$ in~\eqref{E:iv-reflect-alp-eps}. Let us start by taking
  $\epsilon \downarrow 0$. It follows that
  \begin{align}\label{E:eps-to-0-reflect}
    \lim_{\epsilon \downarrow 0} 2V_{\alpha}(\xi)/\epsilon =
    \begin{dcases}
      0,      & \text{if }\xi \in K_{\alpha}, \\
      \infty, & \text{if } \xi \in (H_0 \setminus K_{\alpha}) \cap C([0,1]).
    \end{dcases}
  \end{align}
  Furthermore, recall that $\mu$ is a Gaussian measure supported on $H_0 \cap
  C([0,1])$. Plugging relation~\eqref{E:eps-to-0-reflect}
  into~\eqref{E:iv-reflect-alp-eps}, it is thus not hard to see that the
  following limit holds true:
  \begin{align*}
    \nu_{\alpha,\epsilon}^G (\ud \xi) \to \nu_{\alpha}^G (\ud \xi) \coloneqq \frac{1}{Z_{\alpha}} e^{2G(\xi) } \nu_{\alpha}(\ud \xi), \quad \text{as $\epsilon \downarrow 0$},
  \end{align*}
  where the measure $\nu_{\alpha}$ and the normalization constant $Z_\alpha$ are
  defined by
  \begin{align}\label{def_nu-alp}
    \nu_{\alpha} (\ud \xi ) \coloneqq \frac{1}{\mu (K_{\alpha})}  \mathbf{1}_{K_{\alpha}} (\xi) \mu (\ud \xi) \quad \text{and} \quad
    Z_{\alpha} \coloneqq \int_{K_{\alpha}} e^{2G(\xi) } \nu_{\alpha}(\ud \xi).
  \end{align}
  The limit in $\alpha$ in relation~\eqref{def_nu-alp} is more involved. It is
  treated separately, and leads to Proposition~\ref{P:3d-bessel} below. This
  proposition asserts that $\nu_{\alpha}$ converges weakly to the law of a
  Bessel bridge. Together with relation~\eqref{def_nu-alp}, this concludes the
  proof of Theorem~\ref{T:inv-reflect}.
\end{proof}

\begin{remark}
  Recall that $\mu$ coincides with the distribution of the $1$-$d$ Brownian
  bridge. Thus by~\cite[Chapter III, Exercise (3.14)]{revuz.yor:91:continuous},
  $\mu (K_{\alpha})$ for $\alpha \ge 0$ can be explicitly computed as follows:
  \begin{align}\label{E:mu-k-alp}
    \mu (K_{\alpha}) = 1 - e^{-2\alpha^2}.
  \end{align}
  Hence $\mu (K_{\alpha})\sim 2\al^{2}$ as $\al\to 0$, which dictates the
  singularity of $\nu_{\al}$ in~\eqref{def_nu-alp}. Sorting out this singularity
  is the content of Proposition~\ref{P:3d-bessel}.
\end{remark}

\begin{proposition}\label{P:3d-bessel}
  The measures $\{\nu_{\alpha} (\ud \xi) ; \, \alpha > 0\}$ converge weakly to
  $\nu (\ud \xi)$, the law of the 3-$d$ Bessel bridge, as $\alpha \downarrow 0$.
\end{proposition}
\begin{proof}
  Proposition~\ref{P:3d-bessel} had already been proved
  in~\cite{durrett.iglehart.ea:77:weak}. Here, we provide an alternative proof
  as presented in~\cite{zambotti:01:reflected}. The proof is based on the next
  theorem quoted from~\cite{biane:86:relations}.

  \begin{theorem}[{\cite[Theorem 1]{biane:86:relations}}]\label{T:biane}
    Let $e = \{e_{\tau}; \tau \in [0,1]\}$ be a 3-$d$ Bessel bridge, and let
    $\zeta$ be a random variable uniformly distributed on $[0,1]$ and
    independent of $e$. Consider a process $\beta = \{\beta_{\tau} ; \tau \in
    [0,1]\}$ defined by
    \begin{align*}
      \beta_{\tau} \coloneqq e_{\tau \oplus \zeta} - e_{\zeta}, \quad \text{with $\tau \oplus \zeta \coloneqq \tau + \zeta\; ({\rm mod}\ 1)$.}
    \end{align*}
    Then $\beta $ is a 1-$d$ Brownian bridge.
  \end{theorem}

  Let us now go back to the analysis for $\nu_{\alpha}$. Thanks
  to~\eqref{def_nu-alp}, one can write, that for a generic $\varphi \in
  C_b(H_0)$,
  \begin{align*}
    \int_{H_0} \varphi(\xi) \nu_{\alpha} (\ud \xi) = \frac{1}{\mu(K_{\alpha})} \int_{K_{\alpha}} \varphi (\xi) \mu (\ud \xi),
  \end{align*}
  where we recall that $\mu$ is the distribution of the $1$-$d$ Brownian bridge.
  Now we apply Theorem~\ref{T:biane} in order to express $\mu$ in terms of the
  distribution $\nu$ of the $3$-$d$ Bessel bridge. We get
  \begin{align}\label{E:int-nu-alpha-0}
    \int_{H_0} \varphi(\xi) \nu_{\alpha} (\ud \xi)
     = & \frac{1}{\mu (K_{\alpha})}\int_0^1\ud r\int_{K_0}\varphi (\xi_{\cdot \oplus r} - \xi_{r} ) \mathbf{1}_{K_{\alpha}} (\xi_{\cdot + r} - \xi_r) \, \nu (\ud \xi).
  \end{align}
  In addition, if $\xi \in K_0$, where $K_0$ is defined as
  in~\eqref{E:def-Kalpha} with $\alpha = 0$, we have $\xi \ge 0$ and $\xi(0) =
  \xi(1) = 0$. Therefore, for any $r\in\ou$ we have
  \begin{align*}
    \xi_{\cdot \oplus r} - \xi_{r} \ge - \alpha \iff \xi_{\cdot} - \xi_r \ge - \alpha \iff \xi_r \le \alpha + \xi_{\cdot} \iff \xi_r \le \alpha.
  \end{align*}
  Plugging this elementary relation into~\eqref{E:int-nu-alpha-0}, we obtain
  \begin{align}\label{E:int-nu-alpha-00}
    \int_{H_0} \varphi(\xi) \nu_{\alpha} (\ud \xi)
     = & \frac{1}{\mu (K_{\alpha})}\int_0^1\ud r\int_{K_0}\varphi (\xi_{\cdot \oplus r} - \xi_{r}) \mathbf{1}_{[0,\alpha]} (\xi_r)
     \, \nu (\ud \xi) .
  \end{align}
  We now proceed by splitting the right hand side of~\eqref{E:int-nu-alpha-00}
  in two terms:
  \begin{align}\label{E:int-nu-alpha}
    \int_{H_0} \varphi(\xi) \nu_{\alpha} (\ud \xi) = \frac{1}{\mu (K_{\alpha})} \big(I_1 (\alpha) + I_2 (\alpha)\big),
  \end{align}
  where we recall from~\eqref{E:mu-k-alp} that $\mu (K_{\alpha}) = 1 -
  e^{-2\alpha^2}$, and where
  \begin{eqnarray}\label{E:int-nu-alpha-i1}
    I_1(\alpha) & \coloneqq & \int_0^{1/2}\ud r\int_{K_0}\varphi (\xi_{\cdot \oplus r} - \xi_{r}) \mathbf{1}_{[0,\alpha]} (\xi_r) \nu (\ud \xi), \\
    I_2(\alpha) & \coloneqq & \int_{1/2}^1\ud r\int_{K_0}\varphi (\xi_{\cdot \oplus r} - \xi_{r}) \mathbf{1}_{[0,\alpha]} (\xi_r) \nu (\ud \xi).
    \label{E:int-nu-alpha-i2}
  \end{eqnarray}
  In the sequel, we will only handle the term $I_1(\alpha)$, while the other
  term $I_2(\alpha)$ can be treated exactly along the same lines. The approach
  to addressing $I_1(\alpha)$ involves conditioning the integration in
  $I_1(\alpha)$ to the values of $\xi_{r}$. To this aim, we recall that
  classical considerations on Bessel bridge show that the density of $\xi_r$ is
  of the form
  \begin{align}\label{E:bessel-density}
    \nu \left(\xi_r \in \ud y\right) = g_{r(1-r)}(y) \ud y, \quad \text{with} \;
    g_{\tau} (y) \coloneqq \sqrt{\frac{2}{\pi \tau^3}} \, y^2 \exp \left(-\frac{y^2}{2 \tau}\right), \ \tau > 0,\: y \ge 0.
  \end{align}
  Denoting by $\delta_{\varphi} (r, y)$ the integration of $\varphi$ with
  respect to $\nu$ conditional on $\xi_r = y$, we end up with
  \begin{align*}
    I_1(\alpha) = \int_0^{1/2}\ud r\int_0^{\alpha} \ud y \, g_{r(1-r)} (y) \delta_{\varphi} (r,y).
  \end{align*}
  Recalling~\eqref{E:bessel-density}, by preforming a change of variable
  $y/\sqrt{r(1 - r)} \mapsto y$, we can deduce that
  \begin{align*}
    I_1(\alpha) = \int_0^{1/2}\ud r\int_0^{\frac{\alpha}{r (1 - r)}} \ud y \, g_{1} (y) \, \delta_{\varphi} \left(r, [r (1 - r)]^{1/2}y\right).
  \end{align*}
  Moreover, $r(1 - r) \le 1/4$ for every $r \in [0,1/2]$. Some elementary
  computations thus imply a new decomposition for $I_1 (\alpha)$ of the form
  \begin{align}\label{E:I1}
    I_1 (\alpha) = I_{1,1}(\alpha) + I_{1,2}(\alpha)
  \end{align}
  where
  \begin{eqnarray}\label{E:I11}
    I_{1,1}(\alpha)  & \coloneqq & \int_0^{1/2}\ud r\int_0^{2\alpha} \ud y\; g_{1} (y) \delta_{\varphi} \left(r, [r(1-r)]^{1/2}y\right), \\
    I_{1,2} (\alpha) & \coloneqq & \int_{2\alpha}^{\infty} \ud y\; g_1 (y) \int_0^{\rho(\alpha, y)}\ud r\; \delta_{\varphi} \left(r, [r(1-r)]^{1/2} y\right),
    \label{E:I12}
  \end{eqnarray}
  and the function $\rho$ being defined as
  \begin{align}\label{E:Rho}
    \rho(\alpha, y) \coloneqq \frac{1}{2} \left(1 - \sqrt{1 - (2\alpha/y)^2}\right).
  \end{align}
  In particular in the term $I_{1,2}(\alpha)$ of the previous decomposition, we
  have exchanged the integration order in $r$ and $y$.

  Let us analyze the term $ I_{1,1}(\alpha)$ in~\eqref{E:I11}. In view
  of~\eqref{E:mu-k-alp} and $g_1(y)$ as defined in~\eqref{E:bessel-density} with
  $\tau = 1$, some easy considerations on definite integrals yield
  \begin{align}\label{E:lim-i11-alp}
      \lim_{\alpha \downarrow 0} \frac{I_{1,1} (\alpha)}{\mu(K_{\alpha})}
    = \lim_{\alpha \downarrow 0} \frac{\sqrt{2/\pi}}{1 - e^{- 2 \alpha^2}} \int_0^{2\alpha} \ud y \, y^2 e^{-y^2/2} \int_0^{1/2}\ud r \delta_{\varphi} \left(r, [r(1-r)]^{1/2}y\right)
    = 0.
  \end{align}
  As far as the term $I_{1,2}(\alpha)$ in~\eqref{E:I12} is concerned, it still
  deserves a detailed analysis for which we can refer
  to~\cite{zambotti:01:reflected}. Let us just mention that, since $\rho
  (\alpha, y)$ as spelled out in~\eqref{E:Rho} is approximately $(\alpha/y)^2$
  whenever $\alpha \to 0$, then
  \begin{align}\label{E:lim-i12-alp}
      \lim_{\alpha\downarrow 0} \frac{I_{1,2} (\alpha)}{\mu (K_{\alpha})}
    = & \lim_{\alpha \downarrow 0} \frac{\sqrt{2/\pi}}{1-e^{-2\alpha^2}} \int_{2\alpha}^{\infty} \ud y\; y^2 e^{-y^2/2} \int_0^{(\alpha/y)^2}\ud r\; \delta_{\varphi} \left(r, [r(1-r)]^{1/2}y\right) \nonumber                                                  \\
    = & \lim_{\alpha \downarrow 0} \left(\frac{\sqrt{2/\pi}}{2\alpha^2} \alpha^2 \int_{2\alpha}^{\infty} \ud y\; e^{-y^2/2} (y/\alpha)^2 \int_0^{(\alpha/y)^2}\ud r\right) \times \lim_{r\downarrow 0} \delta_{\varphi} \big(r, [r(1-r)]^{1/2}y\big) \nonumber \\
    = & \frac{1}{2} \int_{K_0} \varphi (\xi) \nu (\ud \xi),
  \end{align}
  where we recall that $\nu$ is the law of the $3$-$d$ Bessel bridge. The last
  identity in relation~\eqref{E:lim-i12-alp} stems from the fact that
  \begin{align*}
      \lim_{r\downarrow 0} \delta_{\varphi} \left(r, [r(1-r)]^{1/2} y\right)
    = \lim_{r \uparrow 1 } \delta_{\varphi} \left(r, [r(1-r)]^{1/2} y\right)
    = \int_{K_0} \varphi(\xi) \nu(\ud \xi).
  \end{align*}
  The detailed proof the aforementioned result can be found in~\cite[Lemma
  6]{zambotti:01:reflected}. Intuitively, this result arises because as $r
  \downarrow 0$ or $\uparrow 1$, the conditioning event $\{\xi_r = \left[r(1 -
  r)\right]^{1/2} y\}$ approaches the entire sample space.

  Summarizing our considerations, plugging~\eqref{E:lim-i11-alp}
  and~\eqref{E:lim-i12-alp} into~\eqref{E:I1}, we find that
  \begin{align}\label{E:int-nu-alpha-i2-fnl}
    \lim_{\alpha \downarrow 0} \frac{I_1(\alpha)}{\mu(K_{\alpha})}
    = \frac{1}{2} \int_{K_0} \varphi (\xi) \nu (\ud \xi).
  \end{align}
  As mentioned above, $I_2(\alpha)$ in~\eqref{E:int-nu-alpha-i2} can be handled
  in the same way. Reporting~\eqref{E:int-nu-alpha-i2-fnl} into the
  decomposition~\eqref{E:int-nu-alpha}, this concludes the proof of
  Theorem~\ref{T:inv-reflect}.
\end{proof}

\begin{remark}
  As seen from our sketch of the proof for Theorem~\ref{T:inv-reflect}, the
  identification of an invariant measure for the stochastic heat equation
  reflected at $0$ hinges upon several key components. First, it relies on
  identifying an invariant measure for the Ornstein-Uhlenbeck process and
  associated gradient systems (this step is borrowed again from Da
  Prato-Zabczyk's series of books, see~\cite[Theorem 6.2.1 and Section
  8.6]{da-prato.zabczyk:96:ergodicity}). Additionally, the relationship between
  the 1-$d$ Brownian bridge conditioning on nonnegative sample paths and the 3-d
  Bessel bridge (see Proposition~\ref{P:3d-bessel}) plays a crucial role, which
  provides a nice representation of the invariant measure as in
  Theorem~\ref{T:inv-reflect}. The invariant measure presented in
  Theorem~\ref{T:inv-reflect} is unique, which can be deduced by comparing the
  asymptotic behaviors of $X_t^{\xi_1}$ and $X_t^{\xi_2}$, which are solutions
  to~\eqref{E:reflect} reflected at $0$ with initial datum $\xi_1$ and $\xi_2$
  respectively. While Zambotti's original work~\cite{zambotti:01:reflected} does
  not explicitly state this uniqueness, it was later affirmed in a broader
  context by Yang and Zhang~\cite[Theorem 2.2]{yang.zhang:14:existence} by
  adopting a coupling method from Mueller~\cite{mueller:93:coupling}. Regarding
  the equation~\eqref{E:reflect}, Zambotti and his collaborators have conducted
  additional studies, such as deriving an integration by parts formula with
  respect to the $3$-$d$ Bessel bridge in~\cite{zambotti:02:integration*1}, and
  exploring fine properties of the contact set $\{(t,x) ; u(t,x) = 0 \}$
  in~\cite{dalang.mueller.ea:06:hitting, zambotti:04:occupation}. In addition to
  the aforementioned works, further investigations into ergodicity and invariant
  measures for stochastic PDEs with reflection have been pursued in various
  scholarly articles, see, e.g., \cite{kalsi:20:existence, otobe:04:invariant,
  zhang:12:large}.
\end{remark}

\section{Ergodicity with degenerate noise}\label{S:NS}

We will now examine a second departure from the Da Prato-Zabczyk standard
setting, which concerns degenerate situations for the noise. Namely we have
mentioned that uniqueness and ergodicity of the invariant measure occurs under
the non-degeneracy assumption~\eqref{E:a1}. In particular, this implies that $W$
should be truly infinite-dimensional (or otherwise stated Rank$(BB^*)=\infty$).

In this section we will describe a celebrated result by Hairer and
Mattingly~\cite{hairer.mattingly:06:ergodicity} in which ergodicity is achieved
in spite of a finite-dimensional noise. The result
in~\cite{hairer.mattingly:06:ergodicity} is proved for stochastic Navier-Stokes
equations, which are well-known systems describing the time evolution of an
incompressible fluid. We consider here a version of the equation for
$x\in\mathbb{R}^2,$ written with the notation of Section~\ref{S:Standard} as
\begin{align}\label{E:NS-Eq}
  \ud u_t+\left(u_t\cdot\nabla\right)u_t = \nu\triangle u_t-\nabla p_t + \xi_t,\quad \mathrm{div}\, u = 0,
\end{align}
where $u_t(x)\in\mathbb{R}^2$ denotes the value of the velocity field at time
$t$ and position $x$, $p_t(x)$ denotes the pressure, and $\xi_t(x)$ is an
external force field acting on the fluid. In the stochastic setting the driving
force $\xi$ is a Gaussian field which is white in time and colored in space.

The existence of an invariant measure for the stochastic PDE~\eqref{E:NS-Eq} can
be proved by some ``soft'' techniques using the regularization and dissipativity
properties of the flow~\cite{cruzeiro:89:solutions, flandoli:94:dissipativity}.
However, showing the uniqueness of the invariant measure is a challenging
problem for at least two reasons:
\begin{enumerate}[wide, labelwidth=!, labelindent=0pt, label=\textbf{(\roman*)}]
  \setlength\itemsep{.05in}

  \item The proof of uniqueness requires a detailed analysis of the
    nonlinearity.

  \item In fact the nonlinearity in~\eqref{E:NS-Eq} can balance the degeneracy
    of $W$ and produces the unique invariant measure. It is thus really at the
    heart of the analysis.

\end{enumerate}
This section is devoted to a brief explanation of those mechanisms. It is
structured as follows:  we will first fix some notation and state the main
result of~\cite{hairer.mattingly:06:ergodicity} in Section~\ref{SS:NS-Setup}.
Section~\ref{SS:NS-Idea} will then be devoted to a brief explanation of the main
idea in the proof.

\subsection{Setup and ergodicity of Navier-Stokes equations}\label{SS:NS-Setup}

the framework used in~\cite{hairer.mattingly:06:ergodicity} is the so-called
vorticity formulation of the Navier-Stokes equation, which will yield a
formulation close to~\eqref{E:general}. Consider~\eqref{E:NS-Eq} on the torus
$\mathbb{T}^2 \coloneqq [-\pi,\pi]^2$ driven by a Gaussian noise $\xi$.  For a
divergence-free velocity field, the vorticity $X$ is defined by $X \coloneqq
\nabla\wedge u = \partial_2 u_1 - \partial_1 u_2$. Note that the velocity and
vorticity formulations are equivalent since $u$ can be recovered from $X$ and
the condition $\nabla\cdot u = 0$. With this notation it can be proved that the
vorticity formulation for the stochastic Navier-Stokes equation is
\begin{align}\label{E:NS-vort}
  \ud X_t = \nu\Delta X_t \, \ud t+B(X,X) \,\ud t+Q \, \ud W_t, \qquad
  X_{0}   = x_{0} \, ,
\end{align}
where $\Delta$ is the Laplacian with periodic boundary conditions and $B(u,w) =
-(u\cdot\nabla) w$ is the usual Navier-Stokes nonlinearity. In what follows, we
will focus on the vorticity formulation of the problem given by
equation~\eqref{E:NS-vort}. Our state space for $X$ will be $H =
L^2_0(\mathbb{T}^2),$ the space of real-valued square-integrable functions on
the torus with vanishing mean.

We now turn to a description of the finite dimensional noise $Q\ud W_t$
in~\eqref{E:NS-vort}. To this aim, we start by introducing a convenient way to
index the Fourier basis of $H$. Namely we write
$\mathbb{Z}^2\backslash\{(0,0)\}=\mathbb{Z}_{+}^2\cup\mathbb{Z}_{-}^2$, where
\begin{align*}
  \mathbb{Z}_-^2 & \coloneqq \left\{(k_1,k_2)\in\mathbb{Z}^2 | -k\in \mathbb{Z}_+^2\right\} \quad\text{and}\quad \\
  \mathbb{Z}_+^2 & \coloneqq \left\{(k_1,k_2)\in\mathbb{Z}^2 | k_2>0\right\}\cup\left\{ (k_1,0)\in\mathbb{Z}^2 | k_1>0\right\}.
\end{align*}
Our set of oscillatory functions is defined in the following way for all
$k\in\mathbb{Z}^2\backslash\{(0,0)\}$: for all $x\in\mathbb{T}^2$ we set
\begin{align}\label{D:f_k}
  f_k(x)= \begin{dcases}
    \sin(k\cdot x), & \mathrm{if}\ k\in\mathbb{Z}^2_+; \\
    \cos(k\cdot x), & \mathrm{if}\ k\in\mathbb{Z}^2_-.
  \end{dcases}
\end{align}
For the remainder of the section, we also fix a set
\begin{align}\label{E:Z0}
  \mathcal{Z}_0=\{k_n;  n=1,\ldots,m\}\subset \mathbb{Z}^2\backslash\{(0,0)\},
\end{align}
which corresponds to the set of driving modes of equation~\eqref{E:NS-vort}. We
are now ready to introduce our finite dimensional noise.

\begin{definition}\label{D:Noise}
  In equation~\eqref{E:NS-vort}, the finite dimensional noise $QW_t$ is defined
  as follows: we start from a $\mathbb{R}^m$-valued Wiener process $W$ defined
  on a probability space $(\Omega, \mathcal{F},\mathbb{P})$. Recall that $f_k$
  is given in~\eqref{D:f_k} and denote by $\{e_n; n=1,\ldots,m\}$ the standard
  basis of $\mathbb{R}^m$.  Define $Q:\mathbb{R}^m\to{H}$ by $Qe_n = q_n
  f_{k_n}$, where the $q_n$ are some strictly positive numbers, and the wave
  numbers $k_n$ are given by the set $\mathcal{Z}_0$. Then for $t \ge 0$ and $x
  \in \mathbb{T}^2$ we have
  \begin{align}\label{E:QW}
    QW_t(x)=\sum_{n=1}^m q_nW^n_tf_{k_n}(x).
  \end{align}
  Notice that $QW_t$ is an element of $H=L^2_0(\mathbb{T}^2)$.
\end{definition}
The main result of~\cite{hairer.mattingly:06:ergodicity}, achieving uniqueness
for the invariant measure in spite of having a finite-dimensional noise
in~\eqref{E:QW}, can now be summarized as follows.

\begin{theorem}\label{T:NS}
  Let $\mathcal{Z}_0$ satisfy the following assumptions:
  \begin{itemize}

    \item[\bf{A1.}] There exist at least two elements in $\mathcal{Z}_0$ with
      different Euclidean norms.

    \item[\bf{A2.}] The set of integer linear combinations of elements of
      $\mathcal{Z}_0$ generates $\mathbb{Z}^2$.

  \end{itemize}
  Then equation~\eqref{E:NS-vort} has a unique invariant measure on ${H}$.
\end{theorem}

\begin{remark}
  Condition \textbf{A2} above is equivalent to the easily verifiable condition
  that the greatest common divisor of the set $\{\det(k,l): k, l \in
  \mathcal{Z}_0\}$ is $1$, where $\det(k,l)$ is the determinant of the
  $2\times2$ matrix with columns $k$ and $l$.
\end{remark}

\begin{example}
  Let $\mathcal{Z}_0=\{\left(1,0\right), (-1,0), (1,1), (-1,-1)\}$. It is clear
  that $\mathcal{Z}_0$ satisfies the assumptions of Theorem~\ref{T:NS}.
  Therefore, equation~\eqref{E:NS-vort} is ergodic with degenerate driving noise
  \begin{align*}
    QW(t,x) = W_1(t)\sin(x_1)
            + W_2(t)\cos(x_1)
            + W_3(t)\sin(x_1+x_2)
            + W_4(t)\cos(x_1+x_2).
  \end{align*}
\end{example}

\subsection{Main idea in the proof of Theorem~\ref{T:NS}}\label{SS:NS-Idea}

The general criterion for uniqueness of the stationary measure is taken again
from Da Prato-Zabczyk's monograph~\cite{da-prato.zabczyk:96:ergodicity}. Namely
it is a well-known and much-used fact that the strong Feller property, combined
with some irreducibility of the transition probability, implies the uniqueness
of the invariant measure. Denote by $P_t$ the underlying semigroup as featured
in~\eqref{E:ave}. In order to establish the strong Feller property of $P_t$, one
usually resorts to an integration by parts argument in the Malliavin calculus
sense, which we now describe.

\subsubsection{Malliavin calculus approach in a non-degenerate case}

Having in mind that the strong Feller property is a crucial step towards
uniqueness of the invariant measure, let us recall a basic criterion allowing to
establish this type of result
(see~\cite[Lemma~7.1.5]{da-prato.zabczyk:96:ergodicity}).
\begin{proposition}\label{P:crt-Feller}
  A semigroup $P_t$ on a Hilbert space $H$ is strong Feller if, for all
  $\varphi: H \to \mathbb{R}$ with $\|\varphi\|_\infty$ and
  $\|\nabla\varphi\|_\infty$ finite and for all $t>0$ one has
  \begin{align}\label{E:crt-Feller}
    \left|\nabla P_t\varphi(x)\right| \le C(\|x\|) \, \|\varphi\|_\infty,
  \end{align}
  where $C:\mathbb{R}_+\to\mathbb{R}$ is a fixed nondecreasing function.
\end{proposition}

Interestingly enough, strong Feller properties can be deduced from Malliavin
calculus considerations. To this aim, let us introduce some additional notation.
First denote by $\Phi_t: C\left([0,t]; \mathbb{R}^m\right) \times H \to H$ the
It\^{o} map such that the solution to~\eqref{E:NS-vort} can be written
$X_t=\Phi_t(W,X_0)$ for every initial condition $X_0 \in H$ and almost every
realization of $W$. Then the infinitesimal variation of $X_t$ with respect to a
perturbation $\xi \in H$ of the initial condition is given by
\begin{align}\label{E:J_t}
  J_t\xi \coloneqq \lim_{\varepsilon\to0}\frac{\Phi_t(W,X_0+\varepsilon\xi)-\Phi_t(W,X_0)}{\varepsilon}.
\end{align}
In addition, the infinitesimal variation of $X_t$ with respect to a perturbation
of the Wiener process in the direction of $V(s) = \int_0^s v(r) \ud r$ with $v
\in L^2([0,T];H)$ is given by the Malliavin derivative
\begin{align}\label{E:D^vX_t}
  \mathcal{D}^v X_t = \lim_{\varepsilon\to0} \frac{\Phi_t(W+\varepsilon V,X_0)-\Phi_t(W,X_0)}{\varepsilon}.
\end{align}
We now state the lemma relating Malliavin derivatives and the strong Feller
property. Notice that it is proved here in a rather informal way.
\begin{lemma}\label{L:crt-Feller}
  Let us assume that for every direction $\xi\in H$ one can find $v = v(\xi)$
  lying in the space $L^2\left([0,T],H\right)$, such that $$\mathcal{D}^v X_t =
  J_t \xi.$$ Then the semigroup $P_{t}$ enjoys the strong Feller property as
  given in Proposition~\ref{P:crt-Feller}.
\end{lemma}

\begin{proof}
  By definition of the semigroup $P_t$, we have
  \begin{align*}
    \langle\nabla{P}_t\varphi(X_0),\xi\rangle
       =\mathbb{E}\left[\nabla\varphi(X_t) \, J_t\xi\right]
       =\mathbb{E}\left[\nabla\varphi(X_t) \, \mathcal{D}^vX_t\right],
  \end{align*}
  where the second identity stems from our assumption $\mathcal{D}^v X_t = J_t
  \xi$. Now it is easily seen that $\nabla \varphi(X_t) \mathcal{D}^v X_t =
  \mathcal{D}^v\left(\varphi(X_t)\right)$. Hence a standard integration by parts
  in the Malliavin calculus sense yields
  \begin{align*}
      \langle\nabla{P}_t\varphi(X_0),\xi\rangle
    = \mathbb{E}\left[\varphi(X_t)\int_0^tv(s)\ud W_s\right].
  \end{align*}
  Therefore we get
  \begin{align*}
    \left| \langle\nabla{P}_t\varphi(X_0),\xi\rangle\right|
    \le \|\varphi\|_\infty\|v\|_{L^2([0,T],H)},
  \end{align*}
  from which inequality~\eqref{E:crt-Feller} is easily deduced. This concludes
  the proof of Lemma~\ref{L:crt-Feller}.
\end{proof}

Thanks to Lemma~\ref{L:crt-Feller}, the strong Feller property for $X_t$ is
reduced to a relation involving its Malliavin derivative. However, the ability
to find a $v$ such that $\mathcal{D}^v X_t = J_t \xi$ relies on some
invertibility of the Malliavin matrix on a proper space. This requirement poses
some major technical difficulty in an infinite dimensional setting. In addition,
a more fundamental question is whether one should expect the strong Feller
property for an infinitely dimensional system at all when the noise is
degenerate. Indeed, it seems that the only result showing the strong Feller
property for an infinite dimensional system where the covariance of the noise
does not have a dense range is given in~\cite{hairer.mattingly:06:ergodicity}.
However, this still requires the forcing to act in a non-degenerate way on a
subspace of finite codimension. A different approach is thus necessary for
degenerate noises like the one described in Definition~\ref{D:Noise}.

\subsubsection{Malliavin calculus approach in the degenerate case}

A key observation of Hairer and Mattingly~\cite{hairer.mattingly:06:ergodicity}
is that the strong Feller property is neither essential nor natural for the
study of ergodicity in dissipative infinite-dimensional systems. To provide an
alternative, they introduced the following weaker \textit{asymptotic strong
Feller} property which is satisfied by the system under consideration and is
sufficient to give ergodicity.

Let $d$ be a pseudo-metric on $\mathcal{X}$. Given two positive finite Borel
measures $\mu_1, \mu_2$ on $\mathcal{X}$ with equal mass, denote by
$\mathcal{C}(\mu_1,\mu_2)$ the set of positive measure on $\mathcal{X}^2$ with
marginals $\mu_1$ and $\mu_2$. Define
\begin{align*}
  \|\mu_1-\mu_2\|_d \coloneqq \inf_{\mu\in\mathcal{C}(\mu_1,\mu_2)}\int_{\mathcal{X}^2}d(x,y)\mu(\ud x,\ud y).
\end{align*}

\begin{definition}\label{D:asy-Feller}
  A Markov transition semigroup $P_t$ on a Polish space $\mathcal{X}$ is called
  asymptotically strong Feller at $x$ if there exists a totally separating
  system of pseudo-metrics $\left\{d_n; n\ge1\right\}$ for $\mathcal{X}$ and a
  sequence $\left\{t_n; n\ge 1\right\}$ with $t_n>0$ such that
  \begin{align*}
    \liminf_{U\in \mathcal{U}_x}\: \limsup_{n\to\infty}\: \sup_{y\in U}\Norm{P_{t_n}(x,\cdot)-P_{t_n}(y,\cdot)}_{d_n}=0,
  \end{align*}
  where $\mathcal{U}_x$ is the collection of all open sets containing $x$. It is
  called asymptotically strong Feller if this property holds at every $x \in
  \mathcal{X}$.
\end{definition}

\begin{remark}
  It is proved in~\cite[Corollary 3.5]{hairer.mattingly:06:ergodicity} that if
  $\{d_n; n\ge 1\}$ is a total separating system of pseudo-metrics for
  $\mathcal{X}$, then $\|\mu_1-\mu_2\|_{\mathrm{TV}} = \lim_{n\to\infty}
  \|\mu_1-\mu_2\|_{d_n}$ for any two positive measures $\mu_1$ and $\mu_2$ with
  equal mass on $\mathcal{X}$.
\end{remark}
\begin{remark}
  Note that when $t_n=t$ for all $n$, the transition probabilities
  $P_t(x,\cdot)$ are continuous in the total variation topology and thus $P_s$
  is strong Feller at times $s\ge t$. In order to show~\eqref{E:NS-vort}
  satisfies the asymptotic strong Feller property, one will take $t_n = n$ as
  will be seen below.
\end{remark}

The following criterion for the asymptotic strong Feller property, generalizing
Proposition~\ref{P:crt-Feller}, is then proposed
in~\cite{hairer.mattingly:06:ergodicity}.

\begin{proposition}\label{P:crt-asy-Feller}
  Let $\left\{t_{n}; n \ge 1\right\}$ and $\left\{\delta_n; n \ge 1\right\}$ be
  two positive sequences with $\{t_n\}$ nondecreasing and $\{\delta_n\}$
  converging to zero. A semigroup $P_t$ on a Hilbert space $\mathcal{H}$ is
  asymptotic strong Feller if, for all $\varphi: \mathcal{H} \to \mathbb{R}$
  with $\|\varphi\|_\infty$ and $\|\nabla\varphi\|_\infty$ finite, we have
  \begin{align}\label{E:crt-asy-Feller}
    \left\lvert\nabla{P}_{t_n}\varphi(x)\right\rvert
    \le C(\|x\|)\left(\|\varphi\|_\infty+\delta_n\|\nabla\varphi\|_\infty\right),
  \end{align}
  where $C:\mathbb{R}_+\to\mathbb{R}$ is a fixed nondecreasing function.
\end{proposition}

\begin{remark}\label{R:IBP-asyt}
  The asymptotic strong Feller property can be related to Malliavin calculus
  notions. The point is that the Malliavin matrix is not invertible in our
  degenerate context. Therefore we are not able to construct a $v \in
  L^2([0,T],\mathbb{R}^m)$ for a fixed time $T$ that produces the same
  infinitesimal shift in the solution as a perturbation $\xi$ in the initial
  condition. Instead, we can construct a $v \in L^2([0,\infty); \mathbb{R}^m)$
  such that an infinitesimal shift of the noise in the direction $v$ provides
  asymptotically the same effect as an infinitesimal perturbation in the
  direction $\xi$. In other words, one has $\| J_t \xi - \mathcal{D}^{v_{0,t}}
  X_t \| \to 0$ as $t \to \infty$, where $v_{0,t}$ is the restriction of $v$ on
  the interval $[0,t]$. Set $\rho_t \coloneqq J_t\xi - \mathcal{D}^{v_{0,t}}
  X_t$. Then one has the approximate integration by parts formula:
  \begin{align}\label{E:approx-iby}
    \langle\nabla{P}_t\varphi(X_0),\xi\rangle \nonumber
    & =\mathbb{E}\left(\left(\nabla\varphi\right)(X_t)J_t\xi\right)                                                                   \\
    & =\mathbb{E}\left((\nabla\varphi)(X_t)\mathcal{D}^{v_{0,t}}X_t\right)+\mathbb{E}\left((\nabla\varphi)(X_t)\rho_t\right)\nonumber \\
    & =\mathbb{E}\left(\varphi(X_t)\int_0^tv(s)\ud W_s\right)+\mathbb{E}\left((\nabla\varphi)(X_t)\rho_t\right),
  \end{align}
  from which it follows that
  \begin{align*}
    \left\lvert\langle\nabla{P}_t\varphi(X_0),\xi\rangle\right\rvert
    & \le\|\varphi\|_\infty \times \mathbb{E} \left(\left|\int_0^tv(s)\ud W_s\right|\right) + \|\nabla\varphi\|_\infty \times \mathbb{E}\left(\|\rho_t\|\right).
  \end{align*}
  If $\mathbb{E}\left(\|\rho_t\|\right)$ is small, one easily
  derives~\eqref{E:crt-asy-Feller} from~\eqref{E:approx-iby}. Notice that the
  explicit construction of $v$ in~\eqref{E:approx-iby} is highly non-trivial and
  rather technical. We refer the interested readers to the original paper for
  details. Instead, we make the following remark regarding the construction of
  $v$.
\end{remark}

\begin{remark}\label{R:v}
  In order to construct a suitable $v$ one needs to better incorporate the
  pathwise smoothing which the dynamics possesses at small scales. Due to the
  degenerate nature of the noise, $v$ will be constructed as a non-adapted
  process. The first term in~\eqref{E:approx-iby} is thus understood as a
  Skorohod integral. In order to provide a good estimate of this term, one has
  to have good control of the ``low modes'' when they are not directly forced by
  the noise and when Girsanov's theorem cannot be used directly (cf.
  Theorem~4.12 in~\cite{hairer.mattingly:06:ergodicity}).  This is the heart of
  the analysis for the structure of the Malliavin matrix for
  equation~\eqref{E:NS-vort}. It exploits the algebraic structure of the
  nonlinearity, which transmits the randomness to the non-directly exited
  unstable directions. This results in an associated diffusion which in the end
  is hypoelliptic (see~\cite{mattingly.pardoux:06:malliavin} for more details).
\end{remark}

With Remarks~\ref{R:IBP-asyt} and~\ref{R:v} in hand, let us conclude by stating
a proposition which yields the asymptotic Feller property for the stochastic
Navier-Stokes equation.

\begin{proposition}\label{P:syp-feller-NS}
  Let $P_t$ be the semigroup related to equation~\eqref{E:NS-vort}. Then for all
  $\eta > 0$, there exist constants $C, \delta>0$ such that for every
  Fr\'{e}chet differentiable function $\varphi$ from $H$ to $\mathbb{R}$ one has
  the bound
  \begin{align}\label{E:NS-asy-Feller}
    \|\nabla{P}_n\vp(X_0)\| \le C\exp\left(\eta\|X_0\|^2\right) \times \left(\|\varphi\|_\infty+\|\nabla\varphi\|_\infty e^{-\delta n}\right),
  \end{align}
  for every $X_0 \in H$ and $n \in \mathbb{N}$. Otherwise stated, the asymptotic
  strong Feller property of Proposition~\ref{P:crt-asy-Feller} holds true for
  equation~\eqref{E:NS-vort} by considering $t_n = n$ and $\delta_n = e^{-\delta
  n}$.
\end{proposition}

We close this section by recalling that the asymptotic strong Feller property
stated in Proposition~\eqref{P:syp-feller-NS} leads to ergodicity. We skip those
considerations for sake of conciseness and we refer
to~\cite{hairer.mattingly:06:ergodicity} for more details. Also notice the
reference~\cite{kuksin.shirikyan:12:mathematics}, based on control type
arguments.

\section{Phase transition in ergodicity in $\R^d$ with $d\ge 3$}\label{S:PAM}

This section is devoted to another very important degenerate case for
equation~\eqref{E:G-mild}. That is, instead of considering a degenerate noise
$\dot W$ in~\eqref{E:G-mild}, we will assume that the diffusion coefficient $B$
can vanish. This situation occurs in particular when considering an important
system called \textit{parabolic Anderson model} (referred to as PAM in the
sequel). Our main message here is that in $\R^d$ when $d\ge 3$, one can observe
phase transitions (in terms of the noise intensity) between a situation where an
invariant measure exists and a very different case with no invariant measure. We
begin with a series of preliminary remarks in Section~\ref{SS:PAM-Prelim-rmk}.
Then we state some phase transition results for the parabolic Anderson model and
related equations in Section~\ref{SS:Phase-Trans}. Eventually we will derive a
new result about convergence in distribution to the invariant measure in
Section~\ref{SS:Gu-Li}.

\subsection{Preliminary remarks}\label{SS:PAM-Prelim-rmk}
So far we have expressed our stochastic PDE's in the Da Prato-Zabczyk infinite
dimensional setting. However, most of the results concerning parabolic Anderson
models use the multiparametric setting for SPDEs popularized by
Walsh~\cite{walsh:86:introduction} and Dalang~\cite{dalang:99:extending,
dalang.frangos:98:stochastic}. The current section will thus use this framework.
The reader is referred to~\cite{dalang.quer-sardanyons:11:stochastic} for a
correspondence between the infinite-dimensional and the multiparametric
settings. With this notational warning in mind, in this section we review
existence results for the invariant measure of the following stochastic heat
equation in $\R^d$ with $d\ge 3$:
\begin{equation}\label{E:SHE}
  \left(\dfrac{\partial}{\partial t}-\dfrac{1}{2}\Delta\right) u(t,x) = b(x,u(t,x))\dot W(t,x),\qquad
  \text{for $x\in \R^d$ and $t>0$}.
\end{equation}
In equation~\eqref{E:SHE}, $b(x,u)$ is uniformly bounded in the first variable
and globally Lipschitz continuous in the second variable, i.e., for some
constants $L_b>0$ and $L_0 \ge 0$,
\begin{equation}\label{E:lipcon}
  \lvert b(x,u)-b(x,v) \rvert \le L_b \lvert u-v \rvert \quad \text{and} \quad
  \lvert b(x,0)\rvert         \le L_0       \quad \text{for all $u,v \in \R$ and $x \in \R^d.$ }
\end{equation}
Notice that our hypothesis~\eqref{E:lipcon} allows the degenerate value
$b(x,0)=0$. In particular, it covers the linear case $b(x,u) = \lambda u$, which
is the aforementioned PAM (see~\cite{carmona.molchanov:94:parabolic}). In this
note, we will specifically focus on the degenerate case:
\begin{align}\label{E:b-vanish-0}
  b(x,0) \equiv 0, \quad \text{for all $x\in\R^d$.}
\end{align}

The noise $\dot{W}(t,x)$ featured in equation~\eqref{E:SHE} is defined in a
standard way within the random field framework for stochastic PDE's.
Specifically, $\dot W$ is defined on a complete probability space $(\Omega,
\mathcal{F}, \mathbb{P})$ with the natural filtration $\{\mathcal{F}_t \}_{t\ge
0}$ generated by the noise, it is a centered Gaussian noise that is white in
time and homogeneously colored in space. Its covariance structure is given by
\begin{align}\label{E:Cor}
  \EE \left[ W\left(\psi\right) W\left(\phi\right) \right]
  = \int_0^\infty \ud s \int_{\R^d} \Gamma(\ud x)(\psi(s,\cdot)*\widetilde{\phi}(s,\cdot))(x),
\end{align}
where $\psi$ and $\phi$ are continuous and rapidly decreasing functions,
$\widetilde{\phi}(x) \coloneqq \phi(-x)$, ``$*$'' refers to the convolution in
the spatial variable, and $\Gamma$ is a nonnegative and nonnegative definite
tempered measure on $\R^d$ that is commonly referred to as the
\textit{correlation measure}. The Fourier transform\footnote{We use the
  following convention and notation for Fourier transform: $\mathcal{F}\psi(\xi)
  = \widehat{\psi}(\xi) \coloneqq \int_{\R^d} e^{-i x\cdot \xi}\psi(x) \ud x$
for any Schwarz test function $\psi\in \mathcal{S}(\R^d)$.} of $\Gamma$ (in the
generalized sense) is also a nonnegative and nonnegative definite tempered
measure, which is usually called the \textit{spectral measure} and is denoted by
$\widehat{f}(\ud \xi)$.  Moreover, in the case where $\Gamma$ has a density $f$,
namely $\Gamma(\ud x) = f(x) \ud x$, we write $\widehat{f}(\ud \xi)$ as
$\widehat{f}(\xi) \ud \xi$. Existence and uniqueness of~\eqref{E:SHE} for
bounded (resp.~rough) initial conditions were established
in~\cite{dalang:99:extending} (resp.~\cite{chen.kim:19:nonlinear}) under
\textit{Dalang's condition}
\begin{align}\label{E:Dalang}
  \Upsilon(\beta) \coloneqq \left(2\pi\right)^{-d} \int_{\R^d} \frac{\widehat{f}(\ud\xi)}{\beta+\lvert\xi\rvert^2}<\infty
  \quad \text{for some (and hence all) $\beta>0$.}
\end{align}
Moreover, existence and uniqueness results in this setting rely on It\^o type
estimates which are elaborations of~\eqref{E:Cor}. Namely, for an adapted
process $v = \{v (t,x); t \ge 0, x \in \R^d\}$, we have
\begin{align}\label{E:inner-v}
  \E \left[\left(\int_s^t \int_{\R^d} v(r,y) W(\ud r,\ud y)\right)^2\right]
  = \int_s^t \ud r \iint_{\R^{2d}} f(y - y') \, \E[v(r,y) v(r,y')] \, \ud y \ud y'.
\end{align}
Ideally one would like to consider general space-time covariance structures like
in~\cite{hu.huang.ea:15:stochastic}. However fractional dependences in time yield
a much more complicated picture in terms of ergodic behavior, as assessed e.g in~\cite{deya.panloup.ea:19:rate}.

In the literature (starting from Da Prato-Zabczyk's contributions summarized in
Section~\ref{SS:Ergodic}), the existence of invariant measures for SHE is often
studied with a drift term and a general second order differential operator
$\mathcal{A}$:
\begin{align}\label{E:Drift}
  \left(\dfrac{\partial}{\partial t}- \mathcal{A} \right) u(t,x)
   = g(x,u(t,x)) + b(x,u(t,x)) \dot W(t,x), \quad \text{ for $x\in \mathcal{O}$, $t>0$}.
\end{align}
In order to control the growth of the solution's moments, both the drift term
$g$ and the differential operator $\mathcal{A}$ (paired with its domain
$\mathcal{O}$ and boundary conditions) have to be sufficiently dissipative to
counterbalance the diffusion part that is governed by the diffusion coefficient
$b$ and the noise correlation function/measure $f$. We claim that the setup as
given in~\eqref{E:SHE} (or the one in Chen and
Eisenberg~\cite{chen.eisenberg:24:invariant}) is among the most challenging
ones. This is because:

\begin{enumerate}[wide, labelwidth=!, labelindent=0pt, label=\textbf{(\roman*)}]
  \setlength\itemsep{.1in}

  \item The semigroup for a bounded domain $\mathcal{O}\subseteq\R^d$ usually
    has stronger contraction properties than that for the whole space $\R^d$.
    This bounded domain assumption is the setup for
    Cerrai~\cite{cerrai:01:second, cerrai:03:stochastic}, Brze\'{z}niak and
    G\c{a}tarek~\cite{brzezniak.gatarek:99:martingale}, and Bogachev and
    R\"ockner~\cite{bogachev.rockner:01:elliptic}. One can find more related
    references from the above three papers. For the random field approach,
    Mueller~\cite{mueller:93:coupling} studied the equation on a torus with the
    operator $\frac{\partial^2}{\partial x^2} - \alpha$, with $\alpha>0$. This
    parameter $\alpha$ provides an exponential dissipative effect; see Theorem
    1.2 (\textit{ibid.}).

\end{enumerate}

In the following, we will focus on works where the spatial domain is the entire
space $\R^d$. There are much fewer works in this setting; notably, we will
comment on the following references~\cite{assing.manthey:03:invariant,
  chen.eisenberg:24:invariant, eckmann.hairer:01:invariant,
gu.li:20:fluctuations, misiats.stanzhytskyi.ea:16:existence,
misiats.stanzhytskyi.ea:20:invariant, tessitore.zabczyk:98:invariant}. Most of
the works in this line have been carried out in some weighted $L^2(\R^d)$ space
with a weight function $\rho$, referred as $L_\rho^2(\R^d)$. This setup was
initially introduced by Tessitore and
Zabczyk~\cite{tessitore.zabczyk:98:invariant}. However, there are two
exceptions: Eckmann and Hairer used the $L^\infty\left(\R\right)$ space and Gu
and Li~\cite{gu.li:20:fluctuations} used the space of continuous functions
denoted by $C\left(\R^d\right)$.

\begin{enumerate}[wide, labelwidth=!, labelindent=0pt, label=\textbf{(\roman*)}, start=2]
  \setlength\itemsep{.1in}

  \item It is, in general, much harder to handle the degenerate diffusion
    coefficient case: $b(u) = u$. Eckmann and
    Hairer~\cite{eckmann.hairer:01:invariant} studied the SHE on $\R$, albeit
    with an additive noise (that is $b(\cdot) \equiv 1$ in~\eqref{E:Drift}). In
    that case, a nonlinear drift term of the form $g(u) = u(1-u^2)$ contributes
    to a sufficient amount of dissipativeness. A different setting is provided
    by Misiats, Stanzhytskyi \textit{et
    al}~\cite{misiats.stanzhytskyi.ea:20:invariant}. Namely, they assume that
    the diffusion coefficient $b$ satisfies the following condition:
    \begin{align*}
      \lvert b(x,u_1) - b(x,u_2)\rvert \le L \, \varphi(x) \lvert u_1 - u_2 \rvert,
    \end{align*}
    for some function $\varphi(x)$ that decays fast enough. In a slightly
    earlier paper of Misiats, Stanzhytskyi \textit{et
    al}~\cite{misiats.stanzhytskyi.ea:16:existence}, they also require $b$ to be
    bounded.

\end{enumerate}

Now we will further concentrate on works that allow the degenerate case $b(u) =
u$ in order to be able to cover the parabolic Anderson model.

\begin{enumerate}[wide, labelwidth=!, labelindent=0pt, label=\textbf{(\roman*)}, start=3]
  \setlength\itemsep{.1in}

  \item Lacking a dissipative drift term makes the problem more challenging.
    Indeed, Assing and Manthey's work~\cite{assing.manthey:03:invariant} allows
    $b(u) = u$. Within their framework, they can afford a space-time white noise
    in $d=1$. However, a trace-class noise (i.e., $Q$ in~\eqref{E:Noise-STD}
    being a trace-class operator) has to be considered in $d\ge 2$. Such general
    setup is achieved by a strong dissipative condition on the drift term $g$ in
    the following form: for some constants $C$ and $\kappa>0$, it holds that
    \begin{align*}
      u g(u) \le C - \kappa u^2, \quad \text{for all $u \in \R$}.
    \end{align*}
    In particular, $g$ cannot vanish.

\end{enumerate}

Now let us further narrow down to works without a drift term. In this case, only
the following papers are left: \cite{chen.eisenberg:24:invariant,
gu.li:20:fluctuations,tessitore.zabczyk:98:invariant}. In all these works,
without the help of the drift term, one has to use the weak dissipative property
of the heat kernel on the whole space coming from the factor $t^{-d/2}$ of the
heat kernel in $\R^d$. If we want this dissipative effect to be strong enough,
this imposes $d\ge 3$.

\begin{enumerate}[wide, labelwidth=!, labelindent=0pt, label=\textbf{(\roman*)}, start=4]
  \setlength\itemsep{.1in}

  \item The analysis of a model like PAM is simplified when one considers a
    trace-class noise in~\eqref{E:SHE}. This is what is assumed in Gu and
    Li~\cite{gu.li:20:fluctuations}: the spatial covariance measure $\Gamma$
    in~\eqref{E:Cor} is of the form $\Gamma(\ud x) = f(x)\ud x$, with
    \begin{align}\label{E:Corr-GuLi}
      f(x) = \int_{\R^d} \phi(x+y)\phi(y) \ud y, \quad \text{for some $\phi\in C_c^\infty(\R^d;\R_+)$} \, .
    \end{align}
    This ensures that the noise is trace-class, since
    $\int_{\R^d}\widehat{f}(\xi)\ud \xi = f(0) =
    \Norm{\phi}_{L^2(\R^d)}^{2}<\infty$. It is much more challenging to include
    more singular noises which are not trace-class. This is achieved in
    Tessitore and Zabczyk~\cite{tessitore.zabczyk:98:invariant}, where the
    conditions on the noise are given by Hypothesis 2.1 (i) and especially (3.4)
    of their paper. However, those conditions are very involved and sometimes
    difficult to check; see Section 5.2 of~\cite{chen.eisenberg:24:invariant}
    for a detailed discussion. In contrast, the work by Chen and
    Eisenberg~\cite{chen.eisenberg:24:invariant} provides the following concise
    condition on the correlation function in order to get a phase transition for
    the existence of an invariant measure:
    \begin{align}\label{E:SpectralCond}
      \int_{\R^d} \frac{\widehat{f}(\ud\xi)}{\lvert \xi \rvert^2 \wedge \lvert \xi \rvert^{2(1-\alpha)}}<\infty,
      \quad \text{for some $\alpha\in \left(4\Upsilon(0)L_b^2,1\right)$,}
    \end{align}
    where $L_b$ is the Lipschitz constant as given in~\eqref{E:lipcon}; see
    Remark 1.2 or (1.19) (\textit{ibid.}). Note that the above
    condition~\eqref{E:SpectralCond} implicitly requires that
    \begin{align}\label{E:Upsilon_0}
      \Upsilon(0) \coloneqq \left(2\pi\right)^{-d} \int_{\R^d} \frac{\widehat{f}(\ud \xi)}{\lvert\xi\rvert^2} <\infty.
    \end{align}
    Note also that the correlation in~\eqref{E:Corr-GuLi} satisfies
    condition~\eqref{E:SpectralCond}, by noticing that $\widehat{\phi}$ is a
    Schwarz function. It is worth pointing out that the trace-class noise
    condition $f(0) <\infty$ neither implies nor is implied by
    condition~\eqref{E:SpectralCond}. In particular, one easily finds some
    noises that are not trace-class but satisfy
    condition~\eqref{E:SpectralCond}. We briefly discuss Bessel and Riesz
    kernels below.

    \vspace{.05in}

  \begin{enumerate}[wide, labelwidth=!, labelindent=1em, label=\textbf{(\alph*)}]
    \setlength\itemsep{.1in}

    \item Let $f_s(\cdot)$ denote the family of \textit{Bessel
      kernels\footnote{Interested readers can refer Section 1.2.1
    of~\cite{grafakos:14:modern} for more details of the Bessel kernel.} with
  parameter $s>0$} on $\R^d$, i.e.,
      \begin{align*}
        f_s(x) = \mathcal{F}^{-1}\left[\frac{1}{\left(1+\left\lvert \xi \right\rvert^2\right)^{s/2}}\right](x),
        \quad \text{for $x\in \R^d$}.
      \end{align*}
      Since $f_s$ is both nonnegative and nonnegative definite, it can be served
      either as correlation function or the spectral measure. On the one hand
      (see Proposition~5.2 of~\cite{chen.eisenberg:24:invariant}), assuming
      $d\ge 3$ and using $f_s$ as the correlation function,
      condition~\eqref{E:Upsilon_0} (resp.~condition~\eqref{E:SpectralCond}) is
      satisfied for all $s>d-2$ (resp. $s>d-2(1-\alpha)$). On the other hand,
      $f_s$ is a trace-class operator if and only if $s > d$ since
      \begin{align*}
        \int_{\R^d} \widehat{f}_s(\xi)\ud \xi
        = \int_{\R^d} \left(1+\left\lvert \xi \right\rvert^2\right)^{-s/2} \ud \xi< \infty
        \quad \Longleftrightarrow \quad s > d.
      \end{align*}
      Hence, all Bessel kernels $f_s(\cdot)$ with $s\in(d-2(1-\alpha),d]$
      provide examples of correlation functions that are not trace-class
      operators but can ensure existence of a nontrivial invariant measure. On
      the other hand, when using the Bessel kernel $f_s$ as the spectral
      measure, as explained in Proposition~5.3
      of~\cite{chen.eisenberg:24:invariant}, assuming again $d\ge 3$,
      conditions~\eqref{E:Upsilon_0} and~\eqref{E:SpectralCond} are satisfied
      for all $s>2$ and for all $s>2(1-\alpha)$, respectively. Hence,
      correlation functions $\widehat{f_s}(x) = \left(1+|x|^2\right)^{-s/2}$
      with $s>2(1-\alpha)$ provide examples of trace-class operators that have
      long-range heavy-tail correlations, in contrast to the finite-range
      correlation in~\eqref{E:Corr-GuLi}. Note that $f_s$, used as correlation
      functions, also provides long-range correlation but with a light tail,
      i.e., some exponential tail.

    \item The Riesz kernel $f(x)= |x|^{-\beta}$ plays a prominent role of in the
      study of the homogeneous Gaussian noise, since it provides heavy-tailed
      correlations and is singular at zero. In particular, it is not a
      trace-class operator. However, it is easy to check that it won't satisfy
      either conditions~\eqref{E:SpectralCond} or~\eqref{E:Upsilon_0}.
      Example~5.10 of~\cite{chen.eisenberg:24:invariant} provides kernel
      functions that have similar behaviors to the Riesz kernel with power
      blow-up near zero at some rate and power decay near infinity at possibly a
      different rate. To be more precise, for any $s_1$ and $s_2\in (0,d)$, the
      \textit{Riesz-type kernel with parameters $(s_1,s_2)$} is defined as a
      combination of the Bessel kernel and its Fourier transform:
      \begin{align}\label{E:Riesz-type}
        f_{s_1,s_2}(x)\coloneqq f_{s_1}(x) + \widehat{f}_{s_2}(x) \quad \text{or equivalently} \quad
        \widehat{f}_{s_1,s_2}(\xi)\coloneqq \widehat{f}_{s_1}(\xi) + f_{s_2}(\xi).
      \end{align}
      It is clear that $f_{s_1,s_2}$ is both nonnegative and nonnegative
      definite. From the properties of the Bessel kernel and its Fourier
      transform, we see that
      \begin{align*}
        f_{s_1,s_2} (x) \asymp
          \begin{cases}
            |x|^{s_1-d} & |x| \to 0, \\
            |x|^{-s_2}  & |x| \to \infty.
          \end{cases}
      \end{align*}
      Assume $d\ge 3$. As explained in Example~5.10
      of~\cite{chen.eisenberg:24:invariant}, condition~\eqref{E:Upsilon_0} holds
      provided $s_1>d-2$ and $s_2>2$ and condition~\eqref{E:SpectralCond} holds
      provided that $s_1>d-2(1-\alpha)$ and $s_2>2(1-\alpha)$. Similar to the
      Bessel kernel case, $f_{s_1,s_2}(x)$ is a trace-class operator if and only
      if $s_1>d$. Let us assume $d\ge 3$. Then, when used as the correlation
      function, the Riesz-type kernel $f_{s_1,s_2}$ with $s_1 \in
      (d-2(1-\alpha),d)$ and $s_2 > 2(1-\alpha)$ provides non trace-class
      operators that both have heavy-tail correlations and can ensure existence
      of a nontrivial invariant measure.

  \end{enumerate}

  \item Due to the importance of the Dirac delta initial
    condition~\cite{amir.corwin.ea:11:probability} and the stationary initial
    condition (i.e., the two-sided Brownian motion as the initial data) pointed
    out in~\cite{bertini.cancrini:95:stochastic}, it would be preferable to be
    able to include more general initial conditions, that are neither bounded at
    one point nor at infinity. Chen and Dalang~\cite{chen.dalang:15:moments}
    introduced the so-called \textit{rough initial conditions} to cover these
    unbounded initial conditions; see also~\cite{chen.huang:19:comparison,
    chen.kim:19:nonlinear}. To be more precise, a rough initial condition $\mu$
    refers to a deterministic, locally finite, regular, signed Borel measure
    that satisfies the following integrability condition at infinity:
    \begin{equation}\label{A:icon}
      \int_{\R^d} \left\lvert\mu\right\rvert(\ud x) \exp\left(-a\lvert x\rvert^2\right) < \infty \quad \text{for all $a>0$},
    \end{equation}
    where $\lvert\mu\rvert = \mu_+ + \mu_-$ and $\mu = \mu_+-\mu_-$ refer to the
    \textit{Hahn decomposition} of the measure $\mu$. The work by Chen and
    Eisenberg~\cite{chen.eisenberg:24:invariant} allows the rough initial
    condition.

\end{enumerate}

\begin{remark}\label{R:GuLi}
  We would like to mention that, in terms of techniques, most of the above works
  are based on the Krylov-Bogoliubov theorem
  (see~\cite{da-prato.zabczyk:96:ergodicity, da-prato.zabczyk:14:stochastic,
  kryloff.bogoliouboff:37:theorie} as well as relations~\eqref{E:ave}
  and~\eqref{E:tight} above). However, the work by Gu and
  Li~\cite{gu.li:20:fluctuations} uses a different argument by shifting the
  initial time to negative infinity. As a result, they obtain a stronger result.
  The convergence there is a strong convergence. The arguments rely critically
  on the stationarity of the solution $u(t,x)$, which is a consequence of the
  constant one initial condition. Nevertheless, Gu and Li pointed out that this
  constant initial condition can be perturbed by an $L^\infty(\R^d) \cap
  L^1(\R^d)$ function; see Remark~3.6 (\textit{ibid.}). In particular, it is not
  clear in their framework whether the constant initial condition can be
  perturbed by a Dirac delta measure. We will go back to this approach in
  Section~\ref{SS:Gu-Li}.
\end{remark}

\subsection{Phase transition for the moments and invariant measures}\label{SS:Phase-Trans}

In this part, let us state precisely the phase transition phenomenon under
conditions~\eqref{E:Upsilon_0} and~\eqref{E:SpectralCond}. Recall that the
solution to~\eqref{E:SHE} is understood as the mild solution:
\begin{align}\label{E:mild}
  u(t,x) = J_0(t,x) + \int_0^t \int_{\R^d} p(t-s,x-y)\, b(y, u(s,y))W(\ud s,\ud y),
\end{align}
where the stochastic integral is interpreted as the Walsh integral
(\cite{dalang:99:extending, walsh:86:introduction}) and $J_0(t,x)$ refers to the
solution to the homogeneous equation, namely,
\begin{align}\label{E:def-J0}
  J_0(t,x) = J_0(t,x;\mu) \coloneqq \int_{\R^d} p(t,x-y)\mu(\ud y) = [p_t*\mu] (x).
\end{align}
Here and throughout the rest of the paper, we use $p(t,x)$, or sometimes
$p_t(x)$, to denote the heat kernel: $p(t,x) = p_t(x) \coloneqq (2\pi
t)^{-d/2}\exp\left(- |x|^2/(2t)\right)$.

The first result asserts that under the cone condition~\eqref{E:Cone} below, the
second moments of the solution exhibit a phase transition depending on the noise
intensity, namely, the values of $L_b$ and $l_p$ in~\eqref{E:lipcon}
and~\eqref{E:Cone}, respectively. Let us first define the \textit{upper and
lower (moment) Lyapunov exponents of order $p$} ($p\ge 2$) by
\begin{align}\label{E:Lypnv-x}
  \overline{m}_p (x) \coloneqq \mathop{\lim\sup}_{t\rightarrow+\infty} \frac{1}{t}\log\E\left(|u(t,x)|^p\right),\quad
  \underline{m}_p(x) \coloneqq \mathop{\lim\inf}_{t\rightarrow+\infty} \frac{1}{t}\log\E\left(|u(t,x)|^p\right).
\end{align}
Then we get the following bounds on the exponents
$\sup_{x\in\R^d}\overline{m}(x)$ and $\inf_{x\in\R^d}\underline{m}(x)$ under
hypothesis~\eqref{E:Upsilon_0}.

\begin{theorem}[Theorem 1.3 of~\cite{chen.kim:19:nonlinear}]\label{T:Phase}
  Let $u$ be the solution to~\eqref{E:SHE} starting from an initial condition
  $\mu$, which is a nonnegative Borel measure on $\R^d$ such that
  \begin{align}\label{E:LinearExpID}
    \int_{\R^d} e^{-\gamma |x|} \mu(\ud x) < \infty \, , \quad \text{for all $\gamma >0$.}
  \end{align}
  Assume $b(x,u) = b(u)$ in~\eqref{E:SHE} satisfies the cone condition:
  \begin{align}\label{E:Cone}
    l_b \coloneqq \inf_{x\in\R} \frac{b(x)}{|x|} > 0.
  \end{align}
  Recall that the covariance of the noise is given by~\eqref{E:Cor} and is
  specified by a correlation measure $f$ satisfying~\eqref{E:Dalang}. Also
  recall that $\Upsilon(0)$ is given by~\eqref{E:Upsilon_0}. Then the following
  holds true:

  \noindent \emph{(i)} If $\Upsilon(0)<\infty$, then for some nonnegative
  constants $0<\underline{\lambda}_c\le \overline{\lambda}_c<\infty$, it holds
  that
  \begin{align}\label{E:Phase}
    \begin{dcases}
      \sup_{x\in\R^d} \overline{m}_2 (x) = 0, & \text{if $L_b < \underline{\lambda}_c$}, \\
      \inf_{x\in\R^d} \underline{m}_2(x) > 0, & \text{if $l_b > \overline{\lambda}_c$},
    \end{dcases}
  \end{align}
  where $L_b$ is the Lipschitz constant as given in~\eqref{E:lipcon}.

  \noindent \emph{(ii)} If $\Upsilon(0)=\infty$, then $u(t,x)$ is fully
  intermittent, i.e., $\overline{m}_1(x)\equiv 0$ and
  $\inf_{x\in\R^d}\underline{m}_2(x)>0$.

  \noindent \emph{(iii)} The following two conditions are equivalent:
  \begin{equation}\label{d1}
    \Upsilon(0)<\infty  \quad \Longleftrightarrow \quad
    d\ge 3 \quad \text{and} \quad \int_{\R^d} \frac{f(z)}{|z|^{d-2}}\ud z <\infty.
  \end{equation}
\end{theorem}

\begin{remark}\label{R:Phase-Significant}
  Moments estimates like those in Theorem~\ref{T:Phase} are an important step
  towards ergodicity, as assessed by Theorem~\ref{T:MomCond}. In the current
  case, relation~\eqref{E:Phase} for $L_{b}\le\underline{\lambda}_c$ is an
  indication that a Krylov-Bogoliubov type argument towards ergodicity can be
  applied. On the other hand, relation~\eqref{E:Phase} for $l_b >
  \overline{\lambda}_c$ rules out any type of ergodic result. This is the
  announced phase transition phenomenon, in terms of coefficients which
  represent \emph{noise intensity}.
\end{remark}

\begin{remark}
  As part of Theorem~\ref{T:Phase}, relation~\eqref{d1} asserts that the phase
  transition phenomenon can only occur in dimension 3 or higher.
\end{remark}

In Theorem~\ref{T:Phase}, we have only considered initial conditions with power
growth at infinity, which are tempered/Schwarz measures. The choice of such
initial conditions is due to the fact that the corresponding solution to the
homogeneous equation $J_0(t,x)$ will only induce trivial contributions to the
Lyapunov exponent. However, if we do not focus only on Lyapunov exponents, the
paper~\cite{chen.kim:19:nonlinear} establishes the following pointwise moment
bounds in case of an initial condition with exponential growth:

\begin{theorem}\label{T:BddMnt}
  Let $u$ be the solution to~\eqref{E:SHE} starting from a nonnegative rough
  initial data $\mu$, namely, $\mu$ is is a nonnegative Borel measure on $\R^d$
  such that~\eqref{A:icon} holds. Suppose that the diffusion coefficient
  $b(x,u)=b(u)$ satisfies the following degenerate condition:
   \begin{equation}\label{E:Cond-b}
        b(0) = 0\;, \quad \text{and} \quad
        |b(u_2)-b(u_1)|\le L_b |u_2-u_1|.
      \end{equation}
  Moreover, suppose that $4 L_b^2 \Upsilon (0) < 1$, namely,
  condition~\eqref{E:Upsilon_0} holds and $L_b<(2\Upsilon(0))^{-1/2}$. Then
  \begin{align}\label{E:BddMnt}
    \E\left[u(t,x)^2\right] \le \frac{J_0^2(t,x)}{L_b^2\left[1-4L_b^2\Upsilon(0)\right]}, \quad \text{for all $t\ge 0$ and $x\in\R^d$},
  \end{align}
  where we recall that $J_0 (t,x)$ is defined by~\eqref{E:def-J0}.
\end{theorem}

\begin{remark}
  From this point onward, we have chosen to work with a coefficient $b$ which
  only depends on the variable $u$. Generalizations to coefficients $b(x,u)$ are
  possible.
\end{remark}

\begin{proof}[Proof of Theorem~\ref{T:BddMnt}]
  The proof uses the contraction coefficient $L_b$ in a Gronwall type lemma,
  very similarly to what will be done in~\eqref{E:b1}. We thus spare the details
  for sake of conciseness. Moreover, this type of estimate has already been
  carried out in~\cite{chen.kim:19:nonlinear}. Namely note that the degenerate
  condition~\eqref{E:Cond-b} implies that $|b(x,u)|\le L_b|u|$ for all $u\in\R$
  uniformly in $x\in\R^d$. Hence, we can invoke part (2) of Theorem 2.4
  in~\cite{chen.kim:19:nonlinear}. After applying (2.10)
  of~\cite{chen.kim:19:nonlinear} with $x=x'$ followed by (2.19)
  of~\cite{chen.kim:19:nonlinear} with $\nu=1$ and $\lambda = L_b$, we see that
  \begin{align*}
    \E\left[u(t,x)^2\right] \le L_b^{-2} J_0^2(t,x) H(t; 2L_b^2),
  \end{align*}
  where $H(t;\gamma)$ is defined in (2.14) of~\cite{chen.kim:19:nonlinear} with
  $y=0$ and $\nu=1$. Then by the second part of Lemma 2.5
  of~\cite{chen.kim:19:nonlinear} with $\nu=1$ and $\gamma = 2 L_b^2$, when
  $\Upsilon(0)<\infty$ and $4L_b^2\Upsilon(0)<1$, then
  \begin{align*}
    H(t; 2L_b^2) \le \frac{1}{1-4L_b^2\Upsilon(0)}, \quad \text{for all $t\ge 0$.}
  \end{align*}
  Combining these two bounds proves Theorem~\ref{T:BddMnt}.
\end{proof}

\begin{remark}\label{R:Phase-Additional}
  Here are some additional comments: (1) The most challenging aspect of the
  paper~\cite{chen.kim:19:nonlinear} is the demonstration that the
  \textit{lower} bound of the Lyapunov exponent is strictly positive, i.e., the
  second relationship in~\eqref{E:Phase}. The difficulty comes from the rough
  initial condition and the lack of Feynman-Kac representations for the moments
  for the nonlinear SHE. (2) By the Burkholder-Gundy-Davis inequality (see the
  version in~\cite[Theorem 1.4]{conus.khoshnevisan:12:on}), one can easily
  extend the boundedness of the second moment to the case of $p$-th moments with
  $p \ge 2$ and $L_0>0$ in~\eqref{E:lipcon}; see Theorem~1.7
  of~\cite{chen.huang:19:comparison} for more details.
\end{remark}

Boundedness of the moments in~\eqref{E:BddMnt} paves the way for the existence
of invariant measures. This was already asserted in Da Prato-Zabczyk's setting,
as recalled in~\eqref{E:MomCond} and Remark~\ref{R:Phase-Significant}. However,
one needs to embed the random field solution in some Hilbert space. The moment
estimates in this setup are a straightforward application of
Theorem~\ref{T:Phase}, which is given by Theorem~\ref{T:MomBdd} below. For any
locally integrable and nonnegative function $\rho:\R^d\to\R_+$, denote by
$L^2_{\rho}(\R^d)$ the Hilbert space of $\rho$-weighted square integrable
functions. We will use $\langle \cdot, \cdot \rangle_{\rho}$ and
$\Norm{\cdot}_{\rho}$ to denote the corresponding inner product and norm,
respectively:
\begin{align}\label{E:WeightedNorm}
  \langle f,g \rangle_{\rho} \coloneqq \int_{\R^d} f(x) g(x) \rho(x) \ud x \, , \quad \text{and} \quad
  \Norm{f}_{\rho} \coloneqq \int_{\R^d} |f(x)|^2 \rho(x) \ud x.
\end{align}
The next theorem states that, under proper conditions on the noise structure and
the intensity of the noise (namely, the value of $L_b$), the time dependence of
the moments of the solution in any weighted $L^2(\R^d)$ space comes solely from
the contribution of the initial condition.

\begin{theorem}[Theorem 1.1 of~\cite{chen.eisenberg:24:invariant} for the degenerate diffusion coefficient case]\label{T:MomBdd}

  Let $u(t,x; \mu)$ be the solution to~\eqref{E:SHE} starting from a rough
  initial condition $\mu$ which satisfies~\eqref{A:icon}. Assume that
  \begin{enumerate}[wide, labelwidth=!, labelindent=0pt, label=\emph{(\roman*)}]
    \setlength\itemsep{.1in}

    \item the diffusion coefficient $b$ satisfies the degenerate
      condition~\eqref{E:Cond-b};

    \item $\rho: \R^d \to \R_+$ is a nonnegative $L^1 (\R^d)$ function;

    \item for all $t > 0$, the initial condition $\mu$ satisfies the condition
      that $\mathcal{G}_{\rho}(t;|\mu|) < \infty$, where
      \begin{equation}\label{E:G}
        \mathcal{G}_{\rho}(t;\mu)
        \coloneqq \int_{\R^d} \left[\left(p(t,\cdot) * \mu\right)(x) \right]^2 \rho(x)\: \ud x;
      \end{equation}

    \item the spectral measure $\widehat{f}$ satisfies~\eqref{E:Upsilon_0},
      i.e., $\Upsilon(0)<\infty$, and the Lipschitz constant $L_b$ is small
      enough so that $4 L_b^2\Upsilon(0) < 1$.

  \end{enumerate}
  Then there exists a unique $L^2 (\Omega)$-continuous solution $u(t,x)$ such
  that for some constant $C>0$, which does not depend on $t$, the following
  holds:
  \begin{equation}\label{E:uInLrho}
    \E\left(\Norm{u(t,\cdot; \mu)}_{\rho}^2\right)
    \le C \mathcal{G}_{\rho}\left(t; |\mu|\right)
    < \infty, \quad \text{for any $t > 0$},
  \end{equation}
  where we recall that the norm $\Norm{\cdot}_\rho$ is given
  by~\eqref{E:WeightedNorm}.
\end{theorem}

Recall that for the existence of an invariant measure, we will need moments
bounded in $t$ (see~\eqref{E:MomCond}). However, the moment upper bound
in~\eqref{E:uInLrho} may still blow up as $t\to\infty$; see Proposition 5.1
of~\cite{chen.eisenberg:24:invariant} for one example. Given the continuity of
$t\to \mathcal{G}_\rho \left(t; |\mu|\right)$, one needs to impose
condition~\eqref{E:InitData} below.

Note that Theorem~\ref{T:MomBdd} above imposes no additional restrictions on the
weight function, as long as it is nonnegative and integrable. However, due to
the non-compactness of the ambient space $\R^d$, in order to extend the heat
semigroup to a $C_0$-semigroup on a weighted $L^2(\R^d)$ space one cannot use
arbitrary weight functions. The notion of admissible weight functions plays an
important role.

\begin{definition}[\cite{tessitore.zabczyk:98:invariant}]\label{D:Rho}
  A function $\rho:\R^d\mapsto \R$ is called an \textup{admissible weight
  function} if it is a strictly positive, bounded, continuous, and $L^1
  (\R^d)$-integrable function such that for all $T > 0$, there exists a constant
  $C_\rho(T)$ satisfying
  \begin{equation}\label{E:aw}
    \big(p(t,\cdot)*\rho(\cdot)\big)(x) \le C_\rho(T) \rho(x)
    \quad \text{for all $t\in[0,T]$ and $x\in\R^d$.}
  \end{equation}
\end{definition}
As proved in Proposition 2.1 of~\cite{tessitore.zabczyk:98:invariant}, the
weighted spaces of Definition~\ref{D:Rho} have nice compactness properties.
Namely, for any admissible functions $\rho$ and $\tilde{\rho}$, if
\begin{align}\label{E:rhorho}
  \int_{\R^d}\dfrac{\rho(x)}{\tilde{\rho}(x)}\ud x < \infty,
\end{align}
then for all $t > 0$, the heat semigroup is compact from
$L_{\tilde{\rho}}(\R^d)$ to $L_\rho(\R^d)$. Canonical choices of the admissible
weight functions include the following examples:
\begin{equation}\label{E:admRho}
  \rho(x)             = \exp(-a|x|)               \quad a>0, \quad \text{and} \quad
  \widetilde{\rho}(x) = \left(1+|x|^a\right)^{-1} \quad a>d.
\end{equation}

Finally, the next theorem integrates various components, including initial data,
the noise structure (via the correlation function), the intensity of the noise
(via the value of $L_b$), and admissible weight functions, to establish the
existence of an invariant measure.

\begin{theorem}[Theorem 1.3 of~\cite{chen.eisenberg:24:invariant} for the degenerate diffusion coefficient case]\label{T:InvMea}

  Let $u(t,x)$ be the solution to~\eqref{E:SHE} starting from a rough initial
  condition $\mu$, namely, $\mu$ is a signed Borel measure on $\R^d$ such
  that~\eqref{A:icon} holds. Assume that
  \begin{enumerate}[wide, labelwidth=!, labelindent=0pt, label=\emph{(\roman*)}]
    \setlength\itemsep{.1in}

    \item \label{it:i-InvMea} $b(x,u)=b(u)$ and it verifies condition~\eqref{E:Cond-b}.

    \item \label{it:ii-InvMea} there are two admissible weight functions $\rho$
      and $\tilde{\rho}$ such that~\eqref{E:rhorho} holds and
      \begin{equation}\label{E:InitData}
        \limsup_{t\to\infty}\: \mathcal{G}_{\tilde{\rho}}(t;|\mu|) < \infty;
      \end{equation}

    \item \label{it:iii-InvMea} the spectral measure $\widehat{f}$
      satisfies~\eqref{E:Upsilon_0}, i.e., $\Upsilon(0)<\infty$ and the
      Lipschitz constant $L_b$ is small enough such that $4 L_b^2 \Upsilon(0)
      <1$;

    \item \label{it:iv-InvMea} for some $\alpha \in
      \left(4L_b^2\Upsilon(0),1\right)$, the spectral measure $\widehat{f}$
      satisfies the following condition
      \begin{align}\label{E:Dalang-ab}
        \Upsilon_\alpha(\beta) \coloneqq \left(2\pi\right)^{-d} \int_{\R^d} \frac{\widehat{f}(\ud\xi)}{\left(\beta+\lvert\xi\rvert^2\right)^{1-\alpha}}<\infty
        \quad \text{for some (hence all) $\beta>0$.}
      \end{align}

  \end{enumerate}
  Then we have that
  \begin{enumerate}[wide, labelwidth=!, labelindent=0pt, label=\emph{(\arabic*)}] \setlength\itemsep{.1in}

    \item for any $\tau>0$, the sequence of laws of
      $\{\mathcal{L}u(t,\cdot;\mu)\}_{t\ge \tau}$ is tight, i.e., for any
      $\epsilon \in (0,1)$, there exists a compact set $\mathscr{K}\subset
      L^2_{\rho}(\R^d)$ such that
      \begin{equation}\label{E:cpt}
        \mathscr{L}u(t,\cdot;\mu)(\mathcal{K}) \ge 1-\epsilon,
        \qquad \text{for all $t\ge \tau>0$};
      \end{equation}

    \item there exists a nontrivial invariant measure for~\eqref{E:SHE}.

  \end{enumerate}
\end{theorem}

\begin{remark}
   Conditions~\emph{\ref{it:ii-InvMea}} and~\emph{\ref{it:iii-InvMea}} in
   Theorem~\ref{T:InvMea} can be more compactly written
   as~\eqref{E:SpectralCond}.  Note that condition~\eqref{E:SpectralCond}
   implicitly implies that $4\Upsilon(0) L_b^2 < 1$, in order to find an
   $\alpha$ such that~\eqref{E:Dalang-ab} is satisfied. Once again, let us
   recall that it has to be interpreted as a small intensity condition on the
   noisy forcing term.
\end{remark}

In the literature, condition~\eqref{E:Dalang-ab} is often called the
\textit{strengthened Dalang's condition}. Since this condition first appeared in
the paper by Sanz-Sol\'e and Sarr\`a~\cite{sanz-sole.sarra:02:holder}, we may
also call it the \textit{Sanz-Sol\'e-Sarr\`a condition}. The conditions on the
structure and intensity of the noise, as specified in
Conditions~\emph{\ref{it:iii-InvMea}} and~\emph{\ref{it:iv-InvMea}} of the above
theorem, can be consolidated into a single
condition---condition~\eqref{E:SpectralCond}.

\begin{remark}\label{R:InvMeasure}
  Note that under the degenerate condition~\eqref{E:Cond-b}, i.e., $b(0) \equiv
  0$, the zero function is an invariant measure for~\eqref{E:SHE}. This
  invariant measure is called the \textit{trivial invariant measure}. Once the
  existence of an invariant measure is obtained via the Bogoliubov theorem (see
  the proof of Theorem~\ref{T:MomCond}), it is routine to show that there exists
  a nontrivial one; see Theorem~4.1 of~\cite{tessitore.zabczyk:98:invariant}.
\end{remark}

\begin{remark}\label{R:Factor64}
  The expression $4L_b^2\Upsilon(0)$ in both Theorems~\ref{T:MomBdd}
  and~\ref{T:InvMea} takes the form of $2^7 L_b^2\Upsilon(0)$
  in~\cite{chen.eisenberg:24:invariant} (see (1.10b) and (1.11) \textit{ibid.}).
  This discrepancy arises from applying the general $p$-th moment bounds (for $p
  \ge 2$), as given in (1.14) of~\cite{chen.huang:19:comparison}, to the second
  moment. These $p$-th moment bounds were derived using the
  Burkholder-Gundy-Davis inequality for cases where $b(x,0)$ may not vanish.
  However, when focusing specifically on the second moment, its application is
  unnecessary. Moreover, the degenerate condition~\eqref{E:Cond-b} further
  simplifies the moment bounds.
\end{remark}

Interested readers may refer to Section~5 of~\cite{chen.eisenberg:24:invariant},
where Theorem~\ref{T:InvMea} is discussed in detail from several perspectives
(initial data, weight functions, correlation functions, and a comparison with
the conditions outlined by Tessitore and
Zabczyk~\cite{tessitore.zabczyk:98:invariant}).

\subsection{The asymptotic behavior via Gu-Li's approach}\label{SS:Gu-Li}

This section is devoted to establishing a new result (to the best of our
knowledge). It establishes convergence in law to the invariant distribution
under the ``\textit{weak disorder}'' condition given
by~\emph{\ref{it:iii-InvMea}} in Theorem~\ref{T:InvMea}. Specifically, suppose
the correlation function $f$ is given by~\eqref{E:Corr-GuLi} and suppose the
diffusion coefficient $b$ depends only on the second argument. Gu and Li proved
in~\cite{gu.li:20:fluctuations} that, when the spatial dimension $d \ge 3$, the
solution $u(t, \cdot)$ to the SHE~\eqref{E:SHE} at time $t$ with a flat initial
condition converges in law to a stationary random field in space as $t \to
\infty$. In this part, we will show that one can extend this result to a broader
class of Gaussian noises (namely, those satisfying
condition~\eqref{E:SpectralCond}), adapting the methods
in~\cite{gu.li:20:fluctuations}. A similar question is also considered
in~\cite{gerolla.hairer.ea:23:fluctuations}, where the spatial correlation of
the driving noise is uniformly bounded with a Riesz type tail. With respect to
the aforementioned references~\cite{gerolla.hairer.ea:23:fluctuations,
gu.li:20:fluctuations} we will also consider a broader class of initial
conditions.

\begin{theorem}\label{T:L2Lim}
  Suppose that conditions~\ref{it:i-InvMea} and~\ref{it:iii-InvMea} of
  Theorem~\ref{T:InvMea} hold. Let the following condition~\ref{it:ii'-InvMea}
  replace condition~\ref{it:ii-InvMea} of Theorem~\ref{T:InvMea}:

  \begin{enumerate}[wide, labelwidth=!, labelindent=0pt, label=\emph{(\roman*')}]
    \setcounter{enumi}{1}
    \setlength\itemsep{.1in}

    \item\label{it:ii'-InvMea} For $t \ge 0$, $x \in \R$ and a Borel measure
      $\mu$ on $\R^d$, let $J_0(t,x;\mu)$ be the solution to the homegenous
      equation (see~\eqref{E:def-J0}). We assume that the following limit
      exists:
      \begin{align}\label{E:InitData-Limit}
        \lim_{t\to\infty} \sup_{x\in\R^d} \left| J_0(t,x;\mu)\right| <\infty.
      \end{align}

  \end{enumerate}

  \noindent We also assume that the nonnegative weight function $\rho$ is an
  element in $L^1(\R^d)$, which does not necessarily need to be admissible (see
  Definition~\ref{D:Rho}). Recall that $u(t,\cdot)$ is the solution
  to~\eqref{E:SHE}. Then the following existence and uniqueness statements hold:

  \begin{enumerate}[wide, labelwidth=!, labelindent=0pt, label=\emph{(\arabic*)}]
    \setlength\itemsep{.1in}

    \item\label{it:L2Lim-1} There exists a random field $Z=\{Z(x);\;x\in\R^d\}$
      such that $Z(\cdot)\in L_\rho^2(\R^d)$ a.s. and
      \begin{equation}\label{E:L2Lim}
        u(t,\cdot) \stackrel{\text{(d)}}{\longrightarrow} Z,
          \quad \text{as $t\to\infty$,}
          \quad \text{in $L_\rho^2(\R^d)$.}
      \end{equation}

    \item\label{it:L2Lim-2} Suppose $u_1$ and $u_2$ are two solutions
      to~\eqref{E:SHE} starting from $\mu_1$ and $\mu_2$, respectively. Assume
      that both $\mu_i$ satisfy~\eqref{E:InitData-Limit}. Let $Z_1$ and $Z_2$ be
      the respective limiting random fields given in part~\ref{it:L2Lim-1}. Then
      we have
      \begin{align}\label{E:L2Unique}
        \lim_{t\to\infty} \sup_{x\in\R^d} \left|J_0(t,x;\mu_1-\mu_2)\right| = 0
        \quad \Longrightarrow \quad Z_1 \stackrel{(d)}{=} Z_2\:,
      \end{align}
      where the notation $Z_1 \stackrel{(d)}{=} Z_2$ denotes that the two random
      fields follow the same law.

  \end{enumerate}
\end{theorem}

Before the proof, we first make some remarks.

\begin{remark}[Perturbation condition]\label{R:Pertubation}
  Recall that in Gu and Li~\cite{gu.li:20:fluctuations}, the initial condition
  needs to be of the form $u(0,x) = \lambda + g(x)$ with $g\in L^1(\R^d)\cap
  L^\infty(\R^d)$; see Remark 3.6 (\textit{ibid.}). The above
  Theorem~\ref{T:L2Lim} relaxes this perturbation term $g$. Indeed, the
  perturbation term $g$ can go well beyond functions in $L^1(\R^d)\cap
  L^\infty(\R^d)$.  To state the most general perturbation type assumption on
  the initial condition, suppose that the rough initial measure $\mu$ can be
  decomposed into two parts $\mu = \mu_0+\mu_1$ such that $\mu_0$ is a rough
  initial measure that satisfies~\eqref{E:InitData-Limit}, and $\mu_1$, another
  rough initial measure, is treated as the perturbation. Here is the
  \textit{perturbation condition}:
  \begin{align}\label{E:Pertubation}
    \lim_{t\to\infty} \sup_{x\in\R^d} \left|J_0(t,x;\mu_1)\right| = 0.
  \end{align}
  Thanks to part~\emph{\ref{it:L2Lim-2}} of Theorem~\ref{T:L2Lim}, provided the
  above perturbation condition~\eqref{E:Pertubation} holds, or equivalently, if
  \begin{align*}
    \lim_{t\to\infty} \sup_{x\in\R^d} \left|J_0(t,x;\mu)  \right|=
    \lim_{t\to\infty} \sup_{x\in\R^d} \left|J_0(t,x;\mu_0)\right|,
  \end{align*}
  then the limiting random field $Z$ follows the same law no matter whether the
  system starts from $\mu$ or $\mu_0$. The perturbation
  condition~\eqref{E:Pertubation} gives a partition of all possible rough
  initial conditions into equivalent classes with respect to the law of the
  asymptotic random field $Z$.
\end{remark}

\begin{example}\label{Ex:Pertubation}
  Following the same setup as Remark~\ref{R:Pertubation}, we provide some
  examples for which the perturbation condition~\eqref{E:Pertubation} is
  satisfied:

  \begin{enumerate}[wide, labelwidth=!, labelindent=0pt, label=\textbf{(\alph*)}]
    \setlength\itemsep{.1in}

    \item $\mu_1$ is such that $|\mu_1|$ is a finite Borel measure. One can
      consider for instance a Dirac delta measure $\mu_1 = \delta_0(x)$, or a
      linear combination of the form $\mu_1 = \delta_{-1}(x) - \delta_1(x)$.
      This is clear because $J_0(t,x;|\mu_1|) \le \left(2\pi
      t\right)^{-d/2}|\mu_1|(\R^d)$, and this last quantity converges to $0$ as
      $t\to\infty$.

    \item $|\mu_1|$ does not need to be a finite measure. Consider for example
      $\mu_1 (\ud x) = |x|^{-\alpha}\ud x$, for a given $\alpha\in (0,d)$. In
      this case, $\mu_1$ is a nonnegative and locally finite measure with total
      variation being infinity. However, the perturbation
      condition~\eqref{E:Pertubation} is still satisfied. This fact is proved in
      Example~5.7 of~\cite{chen.eisenberg:24:invariant},
      \begin{align*}
        \sup_{x\in\R^d} J_0(t,x;|\cdot|^{-\alpha})
        \le C_\alpha t^{-\alpha/2} \to 0, \quad
        \text{as $t\to\infty$, with $C_\alpha\coloneqq 2^{-\alpha/2}\Gamma((d-\alpha)/2)\Gamma(d/2)$.}
      \end{align*}

    \item Let $\mu_1 = \sum_{k\in \mathbb{Z}^d} \delta_{2\pi k}(x) -
      (2\pi)^{-d}$. We claim that $\mu_1$ satisfies the perturbation
      condition~\eqref{E:Pertubation}. Indeed, in this case, for all $x\in\R^d$
      and $t> 0$,
      \begin{align*}
        J_0(t,x;\mu_1) = G(t,x) - (2\pi)^{-d}, \quad \text{with $G(t,x)\coloneqq \sum_{k\in \mathbb{Z}^d} p(t,x+2\pi k)$.}
      \end{align*}
      Since $J_0(t,x;\mu_1)$ is $2\pi$ periodic in each direction of $x$, we see
      that for $t\ge 1$,
      \begin{align*}
        \sup_{x\in\R^d}\left|J_0(t,x;\mu_1)\right|
        = \sup_{x\in[-\pi,\pi]^d} \left|G(t,x) - (2\pi)^{-d}\right|
        \le C e^{-t /2} \to 0, \quad \text{as $t \to \infty$,}
      \end{align*}
      where the inequality follows from, e.g., part (4) of Lemma 2.1
      in~\cite{chen.ouyang.ea:23:parabolic}. This proves the claim. As a
      consequence, the solution $u(t,x)$ to~\eqref{E:SHE} starting from the flat
      initial condition $(2\pi)^{-d}$ and from $\sum_{k\in \mathbb{Z}^d}
      \delta_{2\pi k}(x)$ both yield the limiting field $Z$ of the same
      distribution.

  \end{enumerate}
\end{example}

\begin{remark}[Comparison with Theorem~\ref{T:InvMea}]
  Assumption~\emph{\ref{it:ii'-InvMea}} of Theorem~\ref{T:L2Lim} is stronger
  than the assumption~\emph{\ref{it:ii-InvMea}} of Theorem~\ref{T:InvMea}.
  However, Theorem~\ref{T:L2Lim} does not require the Sanz-Sol\'e-Sarr\`a
  condition~\eqref{E:Dalang-ab}, nor does it impose additional requirements on
  the weight function $\rho$ beyond being nonnegative and integrable.
  In~\cite{chen.huang.ea:19:dense}, the authors constructed a correlation
  function $f(x)$ in $\R^3$ that does not satisfy the Sanz-Sol\'e-Sarr\`a
  condition~\eqref{E:Dalang-ab} by applying a logarithmic correction to the
  critical case. They established solutions with unbounded oscillations.
  Nevertheless, the correlation constructed there fulfills the condition
  $\Upsilon(0)<\infty$; see Theorem~2.1 and its proof \textit{ibid}.
\end{remark}

\begin{remark}
  Hypothesis~\emph{\ref{it:iii-InvMea}} in Theorem~\ref{T:InvMea} has been named
  \emph{weak disorder condition} at the beginning of Section~\ref{SS:Gu-Li}.
  This type of condition is related to the random polymer literature and leads
  to central limit type theorems for the polymer measure; see, e.g.,
  \cite{comets:17:directed}. We plan on pursuing this type of result for general
  environments $W$ in a subsequent publication.
\end{remark}

Before proving Theorem \ref{T:L2Lim}, let us introduce some additional notation.

\begin{notation}
  The function $J_0(t,x;|\mu|)$ defined by~\eqref{E:def-J0} is smooth on
  $(0,\infty) \times \R^d$. Hence, condition~\eqref{E:InitData-Limit} implies
  that the following two constants are finite:
  \begin{align}\label{E:Cmu}
    C_\mu  \coloneqq \sup_{x\in \R^d} J_0\left(1,x;|\mu|\right) < \infty \, ,
    \quad \text{and} \quad
    \widehat{C}_\mu \coloneqq \sup_{(t,x)\in [1,\infty) \times \R^d} J_0\left(t,x;|\mu|\right) < \infty.
  \end{align}
  It is clear that $C_\mu \le \widehat{C}_\mu$. For all $t\ge 0$, we also define
  \begin{align}\label{E:k}
    k(t) \coloneqq \iint_{\R^{2d}} \ud y\ud y'\:  p(t,y)p(t,y')f(y-y')
    =  (2\pi)^{-d} \int_{\R^d}\widehat{f}(\ud \xi)\: e^{-t|\xi|^2},
  \end{align}
  where the second equality is an application of the Plancherel theorem.
  Eventually, we set
  \begin{align}\label{E:H}
    H(t) \coloneqq \int_t^\infty k(s) \ud s.
  \end{align}
\end{notation}

We also label a preliminary result for further use. Its elementary proof is left
to the reader for sake of conciseness.

\begin{lemma}\label{L:H}
  Assume that the assumption~\eqref{E:Upsilon_0} holds, i.e.,
  $\Upsilon(0)<\infty$. Then the function $H(t)$ defined in~\eqref{E:H} is
  monotone decreasing, such that
  \begin{align}\label{E:Limits-H}
    H(t) \downarrow 0 \quad \text{as $t\to\infty$}, \quad \text{and} \quad
    H(t) \uparrow (2\pi)^{-d} \int_{\R^{d}} \frac{\widehat{f}( \ud \xi )}{|\xi|^2} = \Upsilon(0), \quad \text{as $t\downarrow 0$.}
  \end{align}
\end{lemma}

We are now ready to prove the main result of this section.

\begin{proof}[Proof of item \ref{it:L2Lim-1} in Theorem~\ref{T:L2Lim}]
  Recall that we assume the coefficient $b$ to depend on the $u$ variable only.
  Thanks to the well-posedness results, as established
  in~\cite{chen.kim:19:nonlinear}, one can find a unique solution
  to~\eqref{E:SHE} starting from the rough initial condition $\mu$. Also note
  that in the proof, we use $\Norm{\cdot}_p$ to denote $L^p(\Omega)$ norms. We
  now divide our proof in several steps. \smallskip

  \noindent \textit{Step 1:~Dynamical system in negative time.} Since we
  consider the large time behavior, one can restart the system at time $1$.
  Indeed, the Cauchy-Schwarz inequality and the moment bound~\eqref{E:BddMnt}
  imply that, for all $t>0$ and $x\in\R^d$,
  \begin{align}\label{E:InitL2Bdd}
    \E & \left(\left[u(1,\cdot)*p(t,\cdot)(x)\right]^2\right) \le \int_{\R^d} \Norm{u(1,y)}_2^2 p(t,x-y)\ud y \notag \\
       & \le \frac{1}{L_b^2\left(1-4L_b^2\Upsilon(0)\right)}\int_{\R^d} J_0^2(1,y) p(t,x-y)\ud y
     \le \frac{C_\mu^2}{L_b^2\left(1-4L_b^2\Upsilon(0)\right)} <\infty,
  \end{align}
  where $C_\mu$ is defined by~\eqref{E:Cmu}. Hence,
  \begin{align*}
    \left[ u(1,\cdot)*p(t,\cdot)\right]\!(x) < \infty \quad
    \text{a.s.~for all $t>0$ and $x\in\R^d$} .
  \end{align*}
  This relation implies that $u(1,\cdot)$, considered as an initial condition,
  satisfies assumption~\eqref{A:icon} a.s. In the following, we will thus
  restart our system after one unit of time. \medskip

  For any parameter $K>0$, we now extend the white noise $\dot{W}$ in time to a
  two-sided noise indexed by $t\in \R$. In addition, define a process $u_K^*$ by
  \begin{align*}
    u_K^* = \left\{u_K^*(t,x):\: (t,x) \in [-K-1, \infty) \times \R^d\right\},
  \end{align*}
  as the unique solution to the following equation
  \begin{align}\label{E:SHE-K*}
  \begin{dcases}
    \left(\dfrac{\partial}{\partial t}-\dfrac{1}{2}\Delta\right) u_K^* (t,x) = b\left(u_K^* (t, x)\right) \dot W(t,x), \quad (t,x) \in [-K-1, \infty) \times \R^d, \\
    u_K^* (-K-1, \cdot) = \mu.
  \end{dcases}
  \end{align}
  We will restart the system at time $-K$, namely,
  \begin{align}\label{E:u_K-K}
    u_K^*(-K,x) = J_0(1,x; \mu) + \int_{-K-1}^{-K}\int_{\R^d} p\left(-K-s,x-y\right) b\left(u_K^*(s,y)\right)W(\ud s,\ud y).
  \end{align}
  According to this procedure, the process $u_K$ is now defined as
  \begin{align*}
    u_K = \left\{u_K(t,x):\: (t,x) \in [-K, \infty) \times \R^d\right\} \, ,
  \end{align*}
  given as the unique solution to the following restarted equation:
  \begin{align}\label{E:SHE-K}
  \begin{dcases}
    \left(\dfrac{\partial}{\partial t}-\dfrac{1}{2}\Delta\right) u_K (t,x) = b\left(u_K (t, x)\right) \dot W_1(t,x), & (t,x) \in [-K, \infty) \times \R^d, \\
    u_K (-K, x) = u_K^*(-K,x),                                                                                       & x\in\R^d, \quad \text{a.s.},
  \end{dcases}
  \end{align}
  where $\dot{W}_1(t,x) = \dot{W}(1+t,x)$ is a time-shifted noise. Note that
  equivalently, $u_K$ is defined through the following mild formulation:
  \begin{align}\label{E:Mild-uK}
    u_K(t,x) = J_0\left(t+K,x; u_K^*(-K,\cdot)\right) + \int_{-K}^t \int_{\R^d} p(t-s,x-y) b(u_K(s,y)) W_1(\ud s, \ud y).
  \end{align}
  Due to the time stationarity of noise $\dot W$, we see that the random field
  \begin{align*}
    \widetilde{u} = \left\{\widetilde{u} (t,x) = u_K \left(t-K-1, x\right); (t,x) \in [0, \infty) \times \R^d\right\}
  \end{align*}
  shares the same distribution as $u$---the solution to~\eqref{E:SHE}. In the
  following, we are thus reduced to prove the following claim in order to
  establish the theorem: \smallskip

  \noindent\textit{Claim:~} $\left\{u_K (0, x); K \ge 0, \;x\in\R^d\right\}$ is
  a Cauchy sequence in $L^\infty\left(\R^d; L^2\left(\Omega\right)\right) \cap
  L^2\left(\Omega; L_\rho^2\left(\R^d\right)\right)$. \smallskip

 \noindent The remainder of the proof is now devoted to prove the above claim.
 \smallskip

  \noindent\textit{Step 2:~Reduction to an $L^2(\Omega)$-estimate.} Let $K > 0$.
  Denote
  \begin{align}\label{E:JK(t,x)}
    J (t + K,x) \coloneqq \sup_{L \ge K}\E \left[ \left(u_{L} (t,x) - u_{K} (t,x)\right)^2\right],
    \quad \text{for all $t \ge -K$ and $x\in\R^d$.}
  \end{align}
  We also define
  \begin{align}\label{E:CalJK(t,x)}
    \mathcal{J}(t + K) \coloneqq \sup_{x\in\R^d} J (t + K,x).
  \end{align}
  Then, the convergence in $L^\infty\left(\R^d;L^2\left(\Omega\right)\right)$,
  as stated in our \emph{Claim} above, holds if and only if, for any finite
  $t\in\R$,
  \begin{align}\label{E:JK->0}
    \lim_{K \to \infty} \mathcal{J} (K + t) = 0.
  \end{align}
  Moreover, relation \eqref{E:JK->0} implies the convergence in
  $L^2_\rho\left(\R^d;L^2\left(\Omega\right)\right)$, as stated in the
  conclusion~\eqref{E:L2Lim} of Theorem~\ref{T:L2Lim}. This implication is
  observed by noticing that
  \begin{align*}
    \int_{\R^d} \E \left[ \left(u_L (0,x) - u_K (0,x)\right)^2\right]\rho(x)\ud x
    \le \mathcal{J} (K) \Norm{\rho}_{L^1(\R^d)} \to 0,\quad \text{as $L>K\to\infty$.}
  \end{align*}
  Hence, the Claim is proved once the limit in~\eqref{E:JK->0} is established.
  \smallskip

  \noindent\textit{Step 3: Some additional notation.} In the following, we will
  focus on proving~\eqref{E:JK->0}. Let us first set up some notation. Since
  $|b(u)|\le L_b|u|$ thanks to~\eqref{E:Cond-b}, we see that the quantity
  $J(t+K,x)$ defined by~\eqref{E:JK(t,x)} can be decomposed as
  \begin{align}\label{E:J-n}
    J (t + K,x) \le 4 I_0^K(t,x) + 4 I_1^K (t,x) + 2 I_2^K (t,x),
  \end{align}
  where
  \begin{equation}\label{E:I_0-d}
  I_0^K(t,x) \coloneqq \sup_{L \ge K}\E\left[\left|J_0\left(t+L,x;u_{L}^*(1,\cdot)\right) - J_0\left(t+K,x;u_{K}^*(1,\cdot)\right)\right|^2\right]  \, ,
  \end{equation}
  \begin{align}\label{E:I_1-e}
    \begin{aligned}
     I_1^K (t,x) \coloneqq L_b^2\: \sup_{L \ge K}\E\Bigg[ \int_{- L}^{-K}\ud s \iint_{\R^{2d}} \ud y \ud y'\: p(t - s, x-y) p(t - s, x-y') &  \\
     \times  |u_{L}(s, y) u_{K}(s, y')| f\left(y - y'\right)                                                                               & \:\Bigg],
    \end{aligned}
  \end{align}
  and
  \begin{align}\label{E:I_2-f}
    \begin{aligned}
     I_2^K (t,x) \coloneqq L_b^2\: \sup_{L \ge K}\E \Bigg[ \int_{- K}^t \ud s \iint_{\R^{2d}} \ud y \ud y'\: p(t - s, x-y ) \left|u_{L}(s,y ) - u_{K}(s,y)\right| &  \\
     \times f(y - y')  p(t - s, x-y') \left|u_{L}(s,y') - u_{K}(s,y')\right|                                                                                      & \:\Bigg].
    \end{aligned}
  \end{align}
   We now bound the terms $I_0^K$, $I_1^K$, and $I_2^K$ separately.
   \smallskip

  \noindent{\textit{Step 4:~Term $I_0^K(t,x)$.}} Let us first bound the term
  $I_0^K$ defined by~\eqref{E:I_0-d}. To this aim, we first recall
  from~\eqref{E:def-J0} that for a given $K > 0$, we have
  \begin{align*}
     J_0 \left(t+K,x; u_{K}^*(-K,\cdot)\right) = \int_{\R^d} p\left(t+K,x-y\right) u_{K}^*(-K,y)\ud y.
  \end{align*}
  Then plug in expression \eqref{E:u_K-K} for $u_K^*$, and use the semigroup
  property for the heat kernel. This allows to write
  \begin{align}\label{E_:J_0}
    J_0 \left(t+K,x; u_{K}^*(-K,\cdot)\right)  = & J_0(t+K+1,x;\mu) + S_K(t,x),
  \end{align}
  where the stochastic integral in~\eqref{E_:J_0} is defined by
  \begin{align}\label{E:Def-SK}
    S_{K}(t,x) \coloneqq \int_{-K-1}^{-K} \int_{\R^d} p\left(t-s,x-y\right) b\left(u_{K}^*(s,y)\right) W(\ud s, \ud y).
  \end{align}
  Reporting this expression into the definition \eqref{E:I_0-d} of $I_0^K
  (t,x)$, some elementary manipulations show that
  \begin{align}\label{E:I_0-g}
    I_0^K(t,x) \le 4 \Theta(t + K)  + 8 \sup_{L \ge K}\Norm{S_{L}(t,x)}_2^2
  \end{align}
  where for all $\tau \ge 0$, the term $\Theta(\tau)$ can be further decomposed
  as
  \begin{align}\label{E:Def-Theta}
    \Theta(\tau) \coloneqq \sup_{s,r \ge \tau}\sup_{x\in\R^d} \left| J_0(s+1,x;\mu) - J_0(r+1,x;\mu)\right|^2.
  \end{align}

  We now proceed to bound the terms in relation \eqref{E:I_0-g}. To begin with,
  it is clear from condition~\eqref{E:InitData-Limit} that the term $\Theta (t)$
  in \eqref{E:Def-Theta} is such that
  \begin{align}\label{E:Theta->0}
    \sup_{\tau\ge 0} \Theta(\tau) < \infty; \quad\text{and}\quad
    \Theta(\tau) \downarrow 0, \text{ as $\tau \to\infty$}.
  \end{align}
  Therefore we immediately get $\lim_{K\to\infty}\Theta(t + K)=0$. Our main task
  is thus to bound the quantity $S_K (t,x)$ introduced in~\eqref{E:Def-SK}.
  However, this term can be estimated thanks to a direct application of
  relation~\eqref{E:inner-v} plus Cauchy-Schwarz's inequality as follows:
  \begin{align*}
    \Norm{S_{K}(t,x)}_2^2
    \le L_b^2 \int_{-K-1}^{-K}\ud s \iint_{\R^{2d}} \ud y \ud y'\:
    & p\left(t-s,x-y\right)  \Norm{u_{K}^*(-K,y)}_2 \\ \times f(y-y')
    & p\left(t-s,x-y'\right) \Norm{u_{K}^*(-K,y')}_2.
  \end{align*}
  Plugging our moment bound~\eqref{E:BddMnt} into the above inequality, we end
  up with
  \begin{align}\label{ieq_sk-l}
    \Norm{S_{K}(t,x)}_2^2 \le \frac{1}{1-4 L_b^2\Upsilon(0)}
    \int_{-K-1}^{-K}\ud s \iint_{\R^{2d}} \ud y  \ud y'\:
    & p\left(t-s,x-y\right)  J_0(1,y;|\mu|) \nonumber \\
      \times f(y-y')
    & p\left(t-s,x-y'\right) J_0(1,y';|\mu|).
  \end{align}
  Therefore, using the constant $C_\mu$ defined in~\eqref{E:Cmu} and the
  function $H(t)$ in \eqref{E:H}, we have
  \begin{align}\label{E:SK-k}
     \sup_{L \ge K}
    & \Norm{S_{L}(t,x)}_2^2 \nonumber\\
    & \le \frac{C_\mu^2}{1-4 L_b^2\Upsilon(0)}
      \sup_{L \ge K}\int_{t+L}^{t+L+1}\ud s \iint_{\R^{2d}} \ud y  \ud y'\:
      p\left(s,x-y\right) p\left(s,x-y'\right) f(y-y')\nonumber\\
    & = \frac{C_\mu^2}{1-4 L_b^2\Upsilon(0)} \sup_{L \ge K}\left[ H(t+L) - H(t+L+1) \right]\nonumber \\
    & \le \frac{C_\mu^2}{1-4 L_b^2\Upsilon(0)}  H(t+K),
  \end{align}
  where the last equality is due to Lemma~\ref{L:H}. Summarizing our
  considerations for the term $I_0^K$, we gather relations \eqref{E:I_0-g},
  \eqref{E:Theta->0}, and \eqref{E:SK-k}, and get
  \begin{align}\label{E_:I_0}
    I_0^K(t,x) \le 4 \Theta (t + K) + \frac{8C_\mu^2H(t+K)}{1-4 L_b^2\Upsilon(0)}.
  \end{align}
  \medskip

  \noindent{\textit{Step 5:~Term $I_1^K(t,x)$.}} For the term $I_1^K$ defined
  by~\eqref{E:I_1-e}, one proceeds similarly to $I_0^K$. That is we start by
  using relation~\eqref{E:inner-v}, which yields
  \begin{align}\label{E:I_1-m}
    I_1^K (t,x) \le L_b^2 \sup_{L \ge K} \int_{-L}^{-K} \ud s \iint_{\R^{2d}}\: \ud y \ud y'
    & p(t-s, x-y)  \Norm{u_{L}(s,y )}_2 \nonumber \\
    \times f(y - y')
    & p(t-s, x-y') \Norm{u_{L}(s,y')}_2 .
  \end{align}
  Next, we notice that $u_{K}(s,y) \coloneqq u(s + K + 1, y)$, where $u$ is the
  solution to~\eqref{E:SHE}. Hence, applying the moment bound~\eqref{E:BddMnt}
  as for~\eqref{E:I_1-e}, we see that
  \begin{align*}
    \Norm{u_{L}(s,y)}_2^2 = \Norm{u(s+L+1,y)}_2^2 \le
    \frac{J_0^2(s+L+1,y;|\mu|)}{L_b^2(1-4L_b^2 \Upsilon(0))}.
  \end{align*}
  Reporting his information into \eqref{E:I_1-m}, we have thus obtained an
  inequality which mimics relation~\eqref{ieq_sk-l}:
  \begin{align*}
    I_1^K (t,x) \le \frac{1}{1-4L_b^2\Upsilon(0)} \sup_{L \ge K}\int_{-L}^{-K} \ud s \iint_{\R^{2d}}\: \ud y \ud y'
    & p(t-s, x-y)  J_0(s+L+1,y;|\mu|) \\ \times f(y - y')
    & p(t-s, x-y') J_0(s+L+1,y';|\mu|) .
  \end{align*}
  From there, we basically repeat the calculations leading to \eqref{E:SK-k}. We
  let the patient reader check that
  \begin{align}\label{E_:I_1}
    I_1^K (t,x) \le \frac{\widehat{C}_\mu^2\: H(t+K)}{1-4L_b^2\Upsilon(0)},
  \end{align}
  where we recall that $\widehat{C}_\mu$ is introduced in \eqref{E:Cmu} and the
  function $H$ is given by~\eqref{E:H}. \medskip

  \noindent\textit{Step 6:~Term $I_2^K(t,x)$.} This term is treated very
  similarly to $I_0^K (t,x)$ and $I_1^K(t,x)$. Let us just summarize briefly
  the computations for the sake of completeness. Specifically, starting from the
  definition of $I_2^K$, applying relation~\eqref{E:inner-v} and the
  Cauchy-Schwarz inequality, we get
  \begin{align*}
    I_2^K (t,x) \le L_b^2\: \sup_{L \ge K} \int_{- K}^t \ud s \iint_{\R^{2d}} \ud y \ud y'\:
                     & p(t - s, y ) \Norm{u_{L}(s,x-y ) - u_{K}(s,x-y )}_2 \\
    \times f(y - y') & p(t - s, y') \Norm{u_{L}(s,x-y') - u_{K}(s,x-y')}_2.
  \end{align*}
  Recalling the notation for $\mathcal{J}(t + K)$ in~\eqref{E:CalJK(t,x)} and
  the definition of $k(s)$ in~\eqref{E:k}, we see that
  \begin{align}\label{E_:I_2}
    I_2^K (t,x) & \le L_b^2\: \int_{- K}^t \ud s\: \mathcal{J}\left(s + K\right) \iint_{\R^{2d}} \ud y \ud y'\:  p(t - s, y ) f(y - y') p(t - s, y') \notag \\
                & =   L_b^2 \int_0^{t + K} \ud s\: \mathcal{J}(t + K - s) k(s) .
  \end{align}

  \noindent{\textit{Step 7:~Conclusion}} In order to conclude the proof
  of~\eqref{E:L2Lim}, combine relations~\eqref{E_:I_0}, \eqref{E_:I_1},
  and~\eqref{E_:I_2} into~\eqref{E:J-n}. Also, recall that $\mathcal{J} (t + K)$
  is defined by~\eqref{E:CalJK(t,x)}. We get that for all $t > K$,
  \begin{equation}\label{E:b1}
    \mathcal{J} (t + K) \le  g_{K}(t) + 2L_b^2\int_0^{t + K} \mathcal{J}(t + K - s) k(s) \ud s,
  \end{equation}
  where the function $g_{K}$ is defined as
  \begin{equation}\label{E:b2}
    g_{K}(t) \coloneqq
    16\Theta(t + K)
    + \frac{4\left(8C_\mu^2 + \widehat{C}_\mu^2\right) H(t + K)}{1-4L_b^2\Upsilon(0)} .
  \end{equation}
  Now recall that $\Theta$ and $H$ are respectively introduced
  in~\eqref{E:Def-Theta} and~\eqref{E:H}. Moreover, owing to~\eqref{E:Theta->0}
  and~\eqref{E:Limits-H}, we have that $g_K$ in~\eqref{E:b2} decreases to 0 as
  $t\to\infty$. Hence one can safely apply~\eqref{E:f-to=0} in order to get
  $\lim_{K \to \infty} \mathcal{J} (t + K) = 0$ for any fixed $t \in \R$
  provided that $4L_b^2 \Upsilon(0) < 1$, which is guaranteed by
  condition~\emph{\ref{it:iii-InvMea}} of Theorem~\ref{T:L2Lim}. This thereby
  confirms~\eqref{E:JK->0}. This proves part~\emph{\ref{it:L2Lim-1}} of
  Theorem~\ref{T:L2Lim}.
\end{proof}

\begin{proof}[Proof of item~\ref{it:L2Lim-2} in Theorem~\ref{T:L2Lim}]
  For $i=1,2$, let $u_i$ be the solution to~\eqref{E:SHE} starting from the
  rough initial condition $\mu_i$ satisfying~\eqref{E:InitData-Limit}.   Let
  $Z_i$ be the respective limiting random fields. Set $u \coloneqq u_1-u_2$,
  $\mu \coloneqq \mu_1 - \mu_2$, $\widetilde{b}(u) \coloneqq b(u+u_2)-b(u_2)$,
  and $Z \coloneqq Z_1 - Z_2$. Assume that
  \begin{align}\label{E_:Unique-mu}
    \lim_{t\to\infty}\sup_{x\in \R^d} \left\lvert J_0(t,x;\mu)\right\rvert = 0.
  \end{align}
  Hence, $u$ is the solution to~\eqref{E:SHE} with the diffusion coefficient $b$
  replaced by $\widetilde{b}$ starting from $\mu$. Note that $\widetilde{b}(0) =
  0$ and $\widetilde{b}$ and $b$ share the same Lipschitz constant $L_b$. Thus,
  the arguments in \textit{Steps~1--7} (proof of Theorem~\ref{T:L2Lim}
  item~\ref{it:L2Lim-1}) still apply for the current $u$. We thus get a limiting
  random field called $Z$. Under this setup, it suffices to prove that $Z
  \stackrel{(d)}{=} 0$, i.e., the limiting random field follows the same law as
  the trivial (vanishing) field. To this aim, we can resort again to a similar
  seven-step proof as for part~\emph{\ref{it:L2Lim-1}}. Below, we only highlight
  the modifications. In particular, \textit{Step~1} remains unchanged, and we
  recall that $u_K$ denotes the solution restarted from negative time $-K$ after
  the system has run for one unit of time from $\mu$. In \textit{Step~2}, all
  occurrences of $u_L$ and suprema over $L\ge K$ should be removed. Notably, the
  quantity $\mathcal{J}(t+K)$ defined in~\eqref{E:CalJK(t,x)} now takes the
  following form
  \begin{equation}\label{E:e1}
    \mathcal{J}(t+K) = \sup_{x\in\R^d} \E\left[u_K^2(t,x)\right].
  \end{equation}
  Proceed similarly in \textit{Step~3}. Now, in the expression of $I_0^K$
  (see\eqref{E:I_0-d}), only one term remains---the one involving $u_K^*$.
  Additionally, $I_1^K \equiv 0$ (see~\eqref{E:I_1-e}). In \textit{Step~4}, the
  upper bound for $\Norm{S_K}_2^2$ as given in~\eqref{E:SK-k} remains valid. The
  expression $\Theta(\tau)$ given in~\eqref{E:Def-Theta} should be replaced by
  \begin{align}\label{E:Def-tilde-Theta}
    \widetilde{\Theta} (\tau) \coloneqq \sup_{s \ge \tau}\sup_{x\in\R^d} \left| J_0(s+1,x;\mu)\right|^2.
  \end{align}
  Condition~\eqref{E_:Unique-mu} implies that both properties
  in~\eqref{E:Theta->0} with $\Theta$ replaced by $\widetilde{\Theta}$ still
  hold. One can skip \textit{Step~5} since $I_1^K\equiv 0$. One can also repeat
  \textit{Step~6} in the proof of item~\emph{\ref{it:L2Lim-1}} above, with
  $\cj(r)$ replaced by the quantity~\eqref{E:e1}. This yields an estimate for
  $I_{2}^{K}$ which is similar to~\eqref{E_:I_2}. Then, in \textit{Step~7} we
  obtain~\eqref{E:b1} and~\eqref{E:b2}, but with $\Theta$ replaced by
  $\widetilde{\Theta}$. Finally, since $\widetilde{\Theta}$ now vanishes as
  $\tau\to\infty$, we conclude in the same way as in \textit{Step~7} that
  $\lim_{K\to\infty} \mathcal{J}(t+K) = 0$. Thus, for any $t\in\R$ and
  $x\in\R^d$ fixed, $u_K(t,x)\to 0$ in $L^2(\Omega)$ as $K\to\infty$ and hence
  the convergence also holds in law. By part~\emph{\ref{it:L2Lim-1}} of
  Theorem~\ref{T:L2Lim}, we also know that $u_K(t,x)\to Z(x)$ in law as
  $K\to\infty$. Therefore, the two limits have to be equal a.s., namely,
  \begin{align*}
     Z(x) = 0, \quad \text{a.s. for all $x\in\R^d$,}
  \end{align*}
  which immediately implies that the finite dimensional distributions of the
  field $Z$ are the same as those of the zero field. This completes the proof of
  part~\emph{\ref{it:L2Lim-2}} of Theorem~\ref{T:L2Lim}.
\end{proof}

\begin{remark}
  Our result in Theorem~\ref{T:L2Lim} is stated as a convergence in
  $L_\rho^{2}(\R^d)$. Under condition~\emph{\ref{it:iv-InvMea}} of
  Theorem~\ref{T:InvMea}, and considering some weighted
  \textit{Garsia-Rodemich-Rumsey} type lemmas, we could have improved the
  topology in which the convergence holds. Namely one can obtain a convergence
  in distribution in weighted H\"older spaces of the form $C_\rho^\beta(\R^d)$,
  for a proper H\"older exponent $\beta$. However, since the use of weighted
  Garsia-Rodemich-Rumsey inequalities is cumbersome, we have postponed this
  objective to a further publication for sake of conciseness.
\end{remark}

\appendix
\section{A Gronwall-type lemma}\label{S:Gronwall}

In this appendix, we prove a generic Gronwall-type lemma, which plays a key role
in the proof of Theorem~\ref{T:L2Lim}.

\begin{lemma}\label{L:Gronwall}
  Let $g(\cdot)$ be a nonnegative and monotone function on $[0,\infty)$ that
  decreases to $0$ as $t\to\infty$, and such that $g(0)<\infty$. Let $k(\cdot)$
  be a nonnegative and integrable function on $[0,\infty)$ and define
  \begin{align}\label{E:def-gnr-H}
    h(t) \coloneqq \int_t^{\infty} k (s) \ud s.
  \end{align}
  Then, we have:

  \begin{enumerate}[wide, labelwidth=!, labelindent=0pt, label=\emph{(\roman*)}]
    \setlength\itemsep{.1in}

    \item\label{it:Int-Hh-i} For any nonnegative integer $n$ and $t\ge 0$, it
      holds that
      \begin{align}\label{E:g*h}
        \int_0^t g \left(\frac{t - s}{2^n}\right) k(s) \ud s
        \le g(0) h \left(\frac{t}{2^{n + 1}}\right) + h(0) g \left(\frac{t}{2^{n + 1}}\right).
      \end{align}
      In particular, when $g = h$ as defined in \eqref{E:def-gnr-H}, the above
      inequality~\eqref{E:g*h} reduces to
      \begin{align}\label{E:h*h}
        \int_0^t h \left(\frac{t - s}{2^n}\right) k(s) \ud s
        \le 2 h(0) h \left(\frac{t}{2^{n + 1}}\right).
      \end{align}

    \item\label{it:Int-Hh-ii} For any nonnegative integer $n$ and $t\ge 0$, it
      holds that
      \begin{align}\label{E:int_Hhn}
        \int_{[0, t]^{n}_<}
            & \ud u_n \cdots \ud u_1\:  g \left(t - u_n \right) \prod_{i = 1}^n k(u_{i} - u_{i - 1} ) \nonumber \\
        \le & \left( 2^{n} - 1\right) g(0) h(0)^{n-1} h \left(\frac{t}{2^{n}}\right) + h(0)^n g \left(\frac{t}{2^{n}}\right),
      \end{align}
      where $[0, t]_<^{n} \coloneqq \left\{(s_1,\dots, s_n) \in [0,t]^n; 0 <
      s_1< \dots < s_n <t\right\}$ and we have used the convention that $u_0
      \equiv 0$.

    \item\label{it:Int-Hh-iii} Let $f$ be a function on $[0,\infty)$ such that
      \begin{align*}
        f(t) \le g(t) + \beta \int_{0}^t f(t - s) k(s) \ud s, \quad \text{for all $t\ge 0$ with $\beta>0$}.
      \end{align*}
      Then, whenever $\beta < [2 h(0)]^{-1}$, the following convergence is true,
      \begin{align}\label{E:f-to=0}
        \lim_{t\to \infty} f(t) = 0.
      \end{align}

  \end{enumerate}

\end{lemma}

\begin{remark}
  This result generalizes \cite[Lemma 3.5]{gu.li:20:fluctuations}, in which $g =
  h$, and $k(t) = 1 \wedge t^{-d/2}$. The difficulty of the above Gronwall-type
  lemma stems from not knowing the exact decay rate of the kernel function
  $k(s)$ as $s\to\infty$, as well as the lack of knowledge regarding whether and
  how $k(s)$ diverges as $s$ approaches zero (see, e.g., the Riesz-type kernel
  given in~\eqref{E:Riesz-type}). The proof below is based on the dominated
  convergence argument. For comparison, when dealing with specific forms of the
  kernel function $k(s)$, such as $k(s) = s^{-1/\alpha}$ for $\alpha>1$,
  explicit computations can be carried out; see Lemma A.2
  of~\cite{chen.hu.ea:21:regularity} for more details.
\end{remark}

\begin{proof}[Proof of Lemma~\ref{L:Gronwall}]
  We will prove the 3 items successively.

  \noindent\emph{Part \ref{it:Int-Hh-i}.} Decompose the integral into two regions:
  \begin{align}\label{E:dec-H}
    \int_0^t g \left(\frac{t - s}{2^n}\right) k(s) \ud s
    = \left( \int_0^{\frac{t}{2^{n+1}}} + \int_{\frac{t}{2^{n+1}}}^t \right)g \left(\frac{t - s}{2^n}\right) k(s) \ud s
    \eqqcolon I_1 + I_2.
  \end{align}
  Noting that both functions $g$ and $h$ are monotonically decreasing on
  $[0,\infty)$, we see that
  \begin{align}\label{E:dec-H2}
    I_1 & \le g \left(\frac{t}{2^{n + 1}}\right) \int_0^{\frac{t}{2^{n+1}}} k (s) \ud s
          \le h(0) g \left(\frac{t}{2^{n + 1}}\right) \quad \text{and}\\[5pt] \label{E:dec-H1}
    I_2 & \le g(0) \int_{\frac{t}{2^{n + 1}}}^{\infty} k(s) \ud s
           =  g(0) h \left(\frac{t}{2^{n + 1}}\right).
  \end{align}
  Plugging~\eqref{E:dec-H2} and~\eqref{E:dec-H1} into~\eqref{E:dec-H} proves
  part~\emph{\ref{it:Int-Hh-i}} of Lemma~\ref{L:Gronwall}. \medskip

  \noindent\emph{Part \ref{it:Int-Hh-ii}.} We will prove~\eqref{E:int_Hhn} by
  induction in $n$. When $n = 1$, \eqref{E:int_Hhn} is simply a consequence of
  part~\emph{\ref{it:Int-Hh-i}}. Assume that $n\ge 2$ and that \eqref{E:int_Hhn}
  holds for $n-1$. For $t> u_1 > 0$, by the induction assumption, we see that
  \begin{align*}
    \int_{[u_1, t]^{n - 1}_<}
        & \ud u_n \cdots \ud u_2\: g \left(t - u_n \right) \prod_{i = 2}^n k(u_{i} - u_{i - 1} )                                    \\
    =   & \int_{[0, t - u_1]^{n - 1}_<} \ud v_n \cdots \ud v_2\: g \left(t - u_1 - v_n \right) \prod_{i = 2}^n k(v_{i} - v_{i - 1}) \\
    \le & \left( 2^{n-1} - 1\right) g(0) h(0)^{n-2} h \left(\frac{t-u_1}{2^{n-1}}\right) + h(0)^{n-1} g \left(\frac{t-u_1}{2^{n-1}}\right).
  \end{align*}
  Denote the left-hand-side (resp.~right-hand-side) of the
  inequality~\eqref{E:int_Hhn} by $L_n(t)$ (resp. $R_n(t)$). Integrating both
  sides of the above inequality from $0$ to $t$ with respect to the $u_1$
  variable gives
  \begin{align*}
    L_n(t)
    \le & \left( 2^{n-1} - 1\right) g(0) h(0)^{n-2} \int_0^t h \left(\frac{t - u_1}{2^{n - 1}}\right) k(u_1) \ud u_1 \\
        & + h(0)^{n-1} \int_0^t g \left(\frac{t - u_1}{2^{n - 1}}\right) k (u_1) \ud u_1
    \le   R_n(t),
  \end{align*}
  where we have applied the bounds~\eqref{E:g*h} and~\eqref{E:h*h} to the above
  two integrals. This proves part~\emph{\ref{it:Int-Hh-ii}} of
  Lemma~\ref{L:Gronwall}. \medskip

  \noindent\emph{Part \ref{it:Int-Hh-iii}.} By iteration, one can write
  \begin{align*}
    f(t) \le & g(t) + \sum_{n = 1}^{\infty} \beta^n \int_0^t \ud s_n \int_{0}^{t  - s_n} \ud s_{n - 1} \cdots \int_0^{t - \sum_{j = 2}^n s_j} \ud s_1\:   g \left(t  - \sum_{j = 1}^n s_j \right) \prod_{i = 1}^n k(s_i) \\
    =        & g(t) + \sum_{n = 1}^{\infty} \beta^n \int_{[0, t]^{n}_<} \ud u_n \cdots \ud u_1 \:  g \left(t - u_n \right) \prod_{i = 1}^n k(u_{i} - u_{i - 1} ).
  \end{align*}
  Thus, part~\emph{\ref{it:Int-Hh-ii}} implies that
  \begin{align} \label{E:f(t)}
    f(t) \le & g (t) + \sum_{n = 1}^{\infty}  \beta^{n} \left[ \left( 2^{n} - 1\right) g(0) h(0)^{n-1} h \left(\frac{t}{2^{n}}\right) + h(0)^n g \left(\frac{t}{2^{n}}\right) \right] \nonumber \\
         \le & g (t) + \max\{1, g(0), h(0)\} \sum_{n = 1}^{\infty}  2\beta [2\beta h(0)]^{n - 1} \left( h \left(\frac{t}{2^{n}}\right)
             + g \left(\frac{t}{2^{n}}\right) \right).
  \end{align}
  By replacing $t$ by $0$ in the last line of~\eqref{E:f(t)} we see that
  provided that $2\beta h(0)<1$, the summation is finite. Hence, one can apply
  the dominated convergence theorem to pass the limit of $t\to\infty$ inside the
  summation. This completes the whole proof of Lemma~\ref{L:Gronwall}.
\end{proof}

\section*{Acknowledgements}
\noindent S. Tindel is supported for this work by NSF grant DMS-2153915. L. Chen
and P. Xia are partially supported by NSF grant DMS-2246850 and L. Chen is
partially supported by a travel grant (MPS-TSM-00959981) from the Simons
foundation.

\bigskip

\bibliographystyle{amsrefs}
\input{ergodic-spdes-12.bbl}

\end{document}

%% file: ergodic-spdes-12.bbl

%% file: ergodic-spdes-12.bbl
\begin{bibdiv}
\begin{biblist}

\bib{amir.corwin.ea:11:probability}{article}{
      author={Amir, Gideon},
      author={Corwin, Ivan},
      author={Quastel, Jeremy},
       title={Probability distribution of the free energy of the continuum
  directed random polymer in {$1+1$} dimensions},
        date={2011},
        ISSN={0010-3640},
     journal={Comm. Pure Appl. Math.},
      volume={64},
      number={4},
       pages={466\ndash 537},
         url={https://doi.org/10.1002/cpa.20347},
      review={\MR{2796514}},
}

\bib{assing.manthey:03:invariant}{article}{
      author={Assing, Sigurd},
      author={Manthey, Ralf},
       title={Invariant measures for stochastic heat equations with unbounded
  coefficients},
        date={2003},
        ISSN={0304-4149},
     journal={Stochastic Process. Appl.},
      volume={103},
      number={2},
       pages={237\ndash 256},
         url={https://doi.org/10.1016/S0304-4149(02)00211-9},
      review={\MR{1950765}},
}

\bib{bensoussan.lions:78:applications}{book}{
      author={Bensoussan, A.},
      author={Lions, J.-L.},
       title={Applications des in\'{e}quations variationnelles en contr\^{o}le
  stochastique},
      series={M\'{e}thodes Math\'{e}matiques de l'Informatique [Mathematical
  Methods of Information Science], No. 6},
   publisher={Dunod, Paris},
        date={1978},
        ISBN={2-04-010336-8},
      review={\MR{513618}},
}

\bib{bertini.cancrini:95:stochastic}{article}{
      author={Bertini, Lorenzo},
      author={Cancrini, Nicoletta},
       title={The stochastic heat equation: {F}eynman-{K}ac formula and
  intermittence},
        date={1995},
        ISSN={0022-4715},
     journal={J. Statist. Phys.},
      volume={78},
      number={5-6},
       pages={1377\ndash 1401},
         url={https://doi.org/10.1007/BF02180136},
      review={\MR{1316109}},
}

\bib{biane:86:relations}{article}{
      author={Biane, Ph.},
       title={Relations entre pont et excursion du mouvement brownien
  r\'{e}el},
        date={1986},
        ISSN={0246-0203},
     journal={Ann. Inst. H. Poincar\'{e} Probab. Statist.},
      volume={22},
      number={1},
       pages={1\ndash 7},
         url={http://www.numdam.org/item?id=AIHPB_1986__22_1_1_0},
      review={\MR{838369}},
}

\bib{bogachev.rockner:01:elliptic}{article}{
      author={Bogachev, Vladimir~I.},
      author={R\"{o}ckner, Michael},
       title={Elliptic equations for measures on infinite-dimensional spaces
  and applications},
        date={2001},
        ISSN={0178-8051},
     journal={Probab. Theory Related Fields},
      volume={120},
      number={4},
       pages={445\ndash 496},
         url={https://doi.org/10.1007/PL00008789},
      review={\MR{1853480}},
}

\bib{brzezniak.gatarek:99:martingale}{article}{
      author={Brze\'{z}niak, Zdzis{\l}aw},
      author={G\c{a}tarek, Dariusz},
       title={Martingale solutions and invariant measures for stochastic
  evolution equations in {B}anach spaces},
        date={1999},
        ISSN={0304-4149},
     journal={Stochastic Process. Appl.},
      volume={84},
      number={2},
       pages={187\ndash 225},
         url={https://doi.org/10.1016/S0304-4149(99)00034-4},
      review={\MR{1719282}},
}

\bib{carmona.molchanov:94:parabolic}{article}{
      author={Carmona, Ren\'{e}~A.},
      author={Molchanov, S.~A.},
       title={Parabolic {A}nderson problem and intermittency},
        date={1994},
        ISSN={0065-9266},
     journal={Mem. Amer. Math. Soc.},
      volume={108},
      number={518},
       pages={viii+125},
         url={https://doi.org/10.1090/memo/0518},
      review={\MR{1185878}},
}

\bib{cerrai:01:second}{book}{
      author={Cerrai, Sandra},
       title={Second order {PDE}'s in finite and infinite dimension},
      series={Lecture Notes in Mathematics},
   publisher={Springer-Verlag, Berlin},
        date={2001},
      volume={1762},
        ISBN={3-540-42136-X},
         url={https://doi.org/10.1007/b80743},
        note={A probabilistic approach},
      review={\MR{1840644}},
}

\bib{cerrai:03:stochastic}{article}{
      author={Cerrai, Sandra},
       title={Stochastic reaction-diffusion systems with multiplicative noise
  and non-{L}ipschitz reaction term},
        date={2003},
        ISSN={0178-8051},
     journal={Probab. Theory Related Fields},
      volume={125},
      number={2},
       pages={271\ndash 304},
         url={https://doi.org/10.1007/s00440-002-0230-6},
      review={\MR{1961346}},
}

\bib{chen.dalang:15:moments}{article}{
      author={Chen, Le},
      author={Dalang, Robert~C.},
       title={Moments and growth indices for the nonlinear stochastic heat
  equation with rough initial conditions},
        date={2015},
        ISSN={0091-1798},
     journal={Ann. Probab.},
      volume={43},
      number={6},
       pages={3006\ndash 3051},
         url={https://doi.org/10.1214/14-AOP954},
      review={\MR{3433576}},
}

\bib{chen.eisenberg:24:invariant}{article}{
      author={Chen, Le},
      author={Eisenberg, Nicholas},
       title={Invariant {M}easures for the {N}onlinear {S}tochastic {H}eat
  {E}quation with {N}o {D}rift {T}erm},
        date={2024},
        ISSN={0894-9840},
     journal={J. Theoret. Probab.},
      volume={37},
      number={2},
       pages={1357\ndash 1396},
         url={https://doi.org/10.1007/s10959-023-01302-4},
      review={\MR{4751295}},
}

\bib{chen.hu.ea:21:regularity}{article}{
      author={Chen, Le},
      author={Hu, Yaozhong},
      author={Nualart, David},
       title={Regularity and strict positivity of densities for the nonlinear
  stochastic heat equation},
        date={2021},
        ISSN={0065-9266},
     journal={Mem. Amer. Math. Soc.},
      volume={273},
      number={1340},
       pages={v+102},
         url={https://doi.org/10.1090/memo/1340},
      review={\MR{4334477}},
}

\bib{chen.huang:19:comparison}{article}{
      author={Chen, Le},
      author={Huang, Jingyu},
       title={Comparison principle for stochastic heat equation on {$\Bbb
  R^d$}},
        date={2019},
        ISSN={0091-1798},
     journal={Ann. Probab.},
      volume={47},
      number={2},
       pages={989\ndash 1035},
         url={https://doi.org/10.1214/18-AOP1277},
      review={\MR{3916940}},
}

\bib{chen.huang.ea:19:dense}{article}{
      author={Chen, Le},
      author={Huang, Jingyu},
      author={Khoshnevisan, Davar},
      author={Kim, Kunwoo},
       title={Dense blowup for parabolic {SPDE}s},
        date={2019},
     journal={Electron. J. Probab.},
      volume={24},
       pages={Paper No. 118, 33},
         url={https://doi.org/10.1214/19-ejp372},
      review={\MR{4029421}},
}

\bib{chen.kim:19:nonlinear}{article}{
      author={Chen, Le},
      author={Kim, Kunwoo},
       title={Nonlinear stochastic heat equation driven by spatially colored
  noise: moments and intermittency},
        date={2019},
        ISSN={0252-9602},
     journal={Acta Math. Sci. Ser. B (Engl. Ed.)},
      volume={39},
      number={3},
       pages={645\ndash 668},
         url={https://doi.org/10.1007/s10473-019-0303-6},
      review={\MR{4066498}},
}

\bib{chen.ouyang.ea:23:parabolic}{article}{
      author={Chen, Le},
      author={Ouyang, Cheng},
      author={Vickery, William},
       title={Parabolic anderson model with colored noise on torus},
        date={2023August},
     journal={Preprint arXiv:2308.10802, to appear in Bernoulli},
         url={http://arXiv.org/abs/2308.10802},
}

\bib{comets:17:directed}{book}{
      author={Comets, Francis},
       title={Directed polymers in random environments},
      series={Lecture Notes in Mathematics},
   publisher={Springer, Cham},
        date={2017},
      volume={2175},
        ISBN={978-3-319-50486-5; 978-3-319-50487-2},
         url={https://doi.org/10.1007/978-3-319-50487-2},
        note={Lecture notes from the 46th Probability Summer School held in
  Saint-Flour, 2016},
      review={\MR{3444835}},
}

\bib{conus.khoshnevisan:12:on}{article}{
      author={Conus, Daniel},
      author={Khoshnevisan, Davar},
       title={On the existence and position of the farthest peaks of a family
  of stochastic heat and wave equations},
        date={2012},
        ISSN={0178-8051},
     journal={Probab. Theory Related Fields},
      volume={152},
      number={3-4},
       pages={681\ndash 701},
         url={https://doi.org/10.1007/s00440-010-0333-4},
      review={\MR{2892959}},
}

\bib{cruzeiro:89:solutions}{article}{
      author={Cruzeiro, Ana~Bela},
       title={Solutions et mesures invariantes pour des \'{e}quations
  d'\'{e}volution stochastiques du type {N}avier-{S}tokes},
        date={1989},
        ISSN={0723-0869},
     journal={Exposition. Math.},
      volume={7},
      number={1},
       pages={73\ndash 82},
      review={\MR{982157}},
}

\bib{da-prato.zabczyk:96:ergodicity}{book}{
      author={Da~Prato, G.},
      author={Zabczyk, J.},
       title={Ergodicity for infinite-dimensional systems},
      series={London Mathematical Society Lecture Note Series},
   publisher={Cambridge University Press, Cambridge},
        date={1996},
      volume={229},
        ISBN={0-521-57900-7},
         url={https://doi.org/10.1017/CBO9780511662829},
      review={\MR{1417491}},
}

\bib{da-prato.zabczyk:02:second}{book}{
      author={Da~Prato, Giuseppe},
      author={Zabczyk, Jerzy},
       title={Second order partial differential equations in {H}ilbert spaces},
      series={London Mathematical Society Lecture Note Series},
   publisher={Cambridge University Press, Cambridge},
        date={2002},
      volume={293},
        ISBN={0-521-77729-1},
         url={https://doi.org/10.1017/CBO9780511543210},
      review={\MR{1985790}},
}

\bib{da-prato.zabczyk:14:stochastic}{book}{
      author={Da~Prato, Giuseppe},
      author={Zabczyk, Jerzy},
       title={Stochastic equations in infinite dimensions},
     edition={Second},
      series={Encyclopedia of Mathematics and its Applications},
   publisher={Cambridge University Press, Cambridge},
        date={2014},
      volume={152},
        ISBN={978-1-107-05584-1},
         url={https://doi.org/10.1017/CBO9781107295513},
      review={\MR{3236753}},
}

\bib{dalang:99:extending}{article}{
      author={Dalang, Robert~C.},
       title={Extending the martingale measure stochastic integral with
  applications to spatially homogeneous s.p.d.e.'s},
        date={1999},
        ISSN={1083-6489},
     journal={Electron. J. Probab.},
      volume={4},
       pages={no. 6, 29},
         url={https://doi.org/10.1214/EJP.v4-43},
      review={\MR{1684157}},
}

\bib{dalang.frangos:98:stochastic}{article}{
      author={Dalang, Robert~C.},
      author={Frangos, N.~E.},
       title={The stochastic wave equation in two spatial dimensions},
        date={1998},
        ISSN={0091-1798},
     journal={Ann. Probab.},
      volume={26},
      number={1},
       pages={187\ndash 212},
         url={https://doi.org/10.1214/aop/1022855416},
      review={\MR{1617046}},
}

\bib{dalang.mueller.ea:06:hitting}{article}{
      author={Dalang, Robert~C.},
      author={Mueller, C.},
      author={Zambotti, L.},
       title={Hitting properties of parabolic s.p.d.e.'s with reflection},
        date={2006},
        ISSN={0091-1798},
     journal={Ann. Probab.},
      volume={34},
      number={4},
       pages={1423\ndash 1450},
         url={https://doi.org/10.1214/009117905000000792},
      review={\MR{2257651}},
}

\bib{dalang.quer-sardanyons:11:stochastic}{article}{
      author={Dalang, Robert~C.},
      author={Quer-Sardanyons, Llu\'{i}s},
       title={Stochastic integrals for spde's: a comparison},
        date={2011},
        ISSN={0723-0869},
     journal={Expo. Math.},
      volume={29},
      number={1},
       pages={67\ndash 109},
         url={https://doi.org/10.1016/j.exmath.2010.09.005},
      review={\MR{2785545}},
}

\bib{deya.panloup.ea:19:rate}{article}{
      author={Deya, Aur\'{e}lien},
      author={Panloup, Fabien},
      author={Tindel, Samy},
       title={Rate of convergence to equilibrium of fractional driven
  stochastic differential equations with rough multiplicative noise},
        date={2019},
        ISSN={0091-1798},
     journal={Ann. Probab.},
      volume={47},
      number={1},
       pages={464\ndash 518},
         url={https://doi.org/10.1214/18-AOP1265},
      review={\MR{3909974}},
}

\bib{durrett.iglehart.ea:77:weak}{article}{
      author={Durrett, Richard~T.},
      author={Iglehart, Donald~L.},
      author={Miller, Douglas~R.},
       title={Weak convergence to {B}rownian meander and {B}rownian excursion},
        date={1977},
        ISSN={0091-1798},
     journal={Ann. Probability},
      volume={5},
      number={1},
       pages={117\ndash 129},
         url={https://doi.org/10.1214/aop/1176995895},
      review={\MR{436353}},
}

\bib{eckmann.hairer:01:invariant}{article}{
      author={Eckmann, Jean-Pierre},
      author={Hairer, Martin},
       title={Invariant measures for stochastic partial differential equations
  in unbounded domains},
        date={2001},
        ISSN={0951-7715},
     journal={Nonlinearity},
      volume={14},
      number={1},
       pages={133\ndash 151},
         url={https://doi.org/10.1088/0951-7715/14/1/308},
      review={\MR{1808628}},
}

\bib{etheridge.labbe:15:scaling}{article}{
      author={Etheridge, Alison~M.},
      author={Labb\'{e}, Cyril},
       title={Scaling limits of weakly asymmetric interfaces},
        date={2015},
        ISSN={0010-3616},
     journal={Comm. Math. Phys.},
      volume={336},
      number={1},
       pages={287\ndash 336},
         url={https://doi.org/10.1007/s00220-014-2243-2},
      review={\MR{3322375}},
}

\bib{flandoli:94:dissipativity}{article}{
      author={Flandoli, Franco},
       title={Dissipativity and invariant measures for stochastic
  {N}avier-{S}tokes equations},
        date={1994},
        ISSN={1021-9722,1420-9004},
     journal={NoDEA Nonlinear Differential Equations Appl.},
      volume={1},
      number={4},
       pages={403\ndash 423},
         url={https://doi.org/10.1007/BF01194988},
      review={\MR{1300150}},
}

\bib{funaki.olla:01:fluctuations}{article}{
      author={Funaki, Tadahisa},
      author={Olla, Stefano},
       title={Fluctuations for {$\nabla\phi$} interface model on a wall},
        date={2001},
        ISSN={0304-4149},
     journal={Stochastic Process. Appl.},
      volume={94},
      number={1},
       pages={1\ndash 27},
         url={https://doi.org/10.1016/S0304-4149(00)00104-6},
      review={\MR{1835843}},
}

\bib{gerolla.hairer.ea:23:fluctuations}{article}{
      author={Gerolla, Luca},
      author={Hairer, Martin},
      author={Li, Xue-Mei},
       title={Fluctuations of stochastic pdes with long-range correlations},
        date={2023March},
     journal={Preprint arXiv:2303.09811},
         url={http://arXiv.org/abs/2303.09811},
}

\bib{grafakos:14:modern}{book}{
      author={Grafakos, Loukas},
       title={Modern {F}ourier analysis},
     edition={Third},
      series={Graduate Texts in Mathematics},
   publisher={Springer, New York},
        date={2014},
      volume={250},
        ISBN={978-1-4939-1229-2; 978-1-4939-1230-8},
         url={https://doi.org/10.1007/978-1-4939-1230-8},
      review={\MR{3243741}},
}

\bib{gu.li:20:fluctuations}{article}{
      author={Gu, Yu},
      author={Li, Jiawei},
       title={Fluctuations of a nonlinear stochastic heat equation in
  dimensions three and higher},
        date={2020},
        ISSN={0036-1410},
     journal={SIAM J. Math. Anal.},
      volume={52},
      number={6},
       pages={5422\ndash 5440},
         url={https://doi.org/10.1137/19M1296380},
      review={\MR{4169750}},
}

\bib{hairer.mattingly:06:ergodicity}{article}{
      author={Hairer, Martin},
      author={Mattingly, Jonathan~C.},
       title={Ergodicity of the 2{D} {N}avier-{S}tokes equations with
  degenerate stochastic forcing},
        date={2006},
        ISSN={0003-486X},
     journal={Ann. of Math. (2)},
      volume={164},
      number={3},
       pages={993\ndash 1032},
         url={https://doi.org/10.4007/annals.2006.164.993},
      review={\MR{2259251}},
}

\bib{hu.huang.ea:15:stochastic}{article}{
      author={Hu, Yaozhong},
      author={Huang, Jingyu},
      author={Nualart, David},
      author={Tindel, Samy},
       title={Stochastic heat equations with general multiplicative {G}aussian
  noises: {H}\"{o}lder continuity and intermittency},
        date={2015},
     journal={Electron. J. Probab.},
      volume={20},
       pages={no. 55, 50},
         url={https://doi.org/10.1214/EJP.v20-3316},
      review={\MR{3354615}},
}

\bib{kalsi:20:existence}{article}{
      author={Kalsi, Jasdeep},
       title={Existence of invariant measures for reflected stochastic partial
  differential equations},
        date={2020},
        ISSN={0894-9840},
     journal={J. Theoret. Probab.},
      volume={33},
      number={3},
       pages={1755\ndash 1767},
         url={https://doi.org/10.1007/s10959-019-00906-z},
      review={\MR{4125974}},
}

\bib{kryloff.bogoliouboff:37:theorie}{article}{
      author={Kryloff, Nicolas},
      author={Bogoliouboff, Nicolas},
       title={La th\'{e}orie g\'{e}n\'{e}rale de la mesure dans son application
  \`a l'\'{e}tude des syst\`emes dynamiques de la m\'{e}canique non
  lin\'{e}aire},
        date={1937},
        ISSN={0003-486X},
     journal={Ann. of Math. (2)},
      volume={38},
      number={1},
       pages={65\ndash 113},
         url={https://doi.org/10.2307/1968511},
      review={\MR{1503326}},
}

\bib{kuksin.shirikyan:12:mathematics}{book}{
      author={Kuksin, Sergei},
      author={Shirikyan, Armen},
       title={Mathematics of two-dimensional turbulence},
      series={Cambridge Tracts in Mathematics},
   publisher={Cambridge University Press, Cambridge},
        date={2012},
      volume={194},
        ISBN={978-1-107-02282-9},
         url={https://doi.org/10.1017/CBO9781139137119},
      review={\MR{3443633}},
}

\bib{mattingly.pardoux:06:malliavin}{article}{
      author={Mattingly, Jonathan~C.},
      author={Pardoux, \'{E}tienne},
       title={Malliavin calculus for the stochastic 2{D} {N}avier-{S}tokes
  equation},
        date={2006},
        ISSN={0010-3640},
     journal={Comm. Pure Appl. Math.},
      volume={59},
      number={12},
       pages={1742\ndash 1790},
         url={https://doi.org/10.1002/cpa.20136},
      review={\MR{2257860}},
}

\bib{misiats.stanzhytskyi.ea:16:existence}{article}{
      author={Misiats, Oleksandr},
      author={Stanzhytskyi, Oleksandr},
      author={Yip, Nung~Kwan},
       title={Existence and uniqueness of invariant measures for stochastic
  reaction-diffusion equations in unbounded domains},
        date={2016},
        ISSN={0894-9840},
     journal={J. Theoret. Probab.},
      volume={29},
      number={3},
       pages={996\ndash 1026},
         url={https://doi.org/10.1007/s10959-015-0606-z},
      review={\MR{3540487}},
}

\bib{misiats.stanzhytskyi.ea:20:invariant}{article}{
      author={Misiats, Oleksandr},
      author={Stanzhytskyi, Oleksandr},
      author={Yip, Nung~Kwan},
       title={Invariant measures for stochastic reaction-diffusion equations
  with weakly dissipative nonlinearities},
        date={2020},
        ISSN={1744-2508},
     journal={Stochastics},
      volume={92},
      number={8},
       pages={1197\ndash 1222},
         url={https://doi.org/10.1080/17442508.2019.1691212},
      review={\MR{4175844}},
}

\bib{mueller:93:coupling}{article}{
      author={Mueller, Carl},
       title={Coupling and invariant measures for the heat equation with
  noise},
        date={1993},
        ISSN={0091-1798},
     journal={Ann. Probab.},
      volume={21},
      number={4},
       pages={2189\ndash 2199},
         url={https://doi.org/10.1214/aop/1176989016},
      review={\MR{1245306}},
}

\bib{nualart.pardoux:92:white}{article}{
      author={Nualart, D.},
      author={Pardoux, \'{E}.},
       title={White noise driven quasilinear {SPDE}s with reflection},
        date={1992},
        ISSN={0178-8051},
     journal={Probab. Theory Related Fields},
      volume={93},
      number={1},
       pages={77\ndash 89},
         url={https://doi.org/10.1007/BF01195389},
      review={\MR{1172940}},
}

\bib{otobe:04:invariant}{article}{
      author={Otobe, Yoshiki},
       title={Invariant measures for {SPDE}s with reflection},
        date={2004},
        ISSN={1340-5705},
     journal={J. Math. Sci. Univ. Tokyo},
      volume={11},
      number={4},
       pages={425\ndash 446},
      review={\MR{2110922}},
}

\bib{revuz.yor:91:continuous}{book}{
      author={Revuz, Daniel},
      author={Yor, Marc},
       title={Continuous martingales and {B}rownian motion},
      series={Grundlehren der mathematischen Wissenschaften [Fundamental
  Principles of Mathematical Sciences]},
   publisher={Springer-Verlag, Berlin},
        date={1991},
      volume={293},
        ISBN={3-540-52167-4},
         url={https://doi.org/10.1007/978-3-662-21726-9},
      review={\MR{1083357}},
}

\bib{sanz-sole.sarra:02:holder}{incollection}{
      author={Sanz-Sol\'{e}, Marta},
      author={Sarr\`a, M\`onica},
       title={H\"{o}lder continuity for the stochastic heat equation with
  spatially correlated noise},
        date={2002},
   booktitle={Seminar on {S}tochastic {A}nalysis, {R}andom {F}ields and
  {A}pplications, {III} ({A}scona, 1999)},
      series={Progr. Probab.},
      volume={52},
   publisher={Birkh\"{a}user, Basel},
       pages={259\ndash 268},
      review={\MR{1958822}},
}

\bib{tessitore.zabczyk:98:invariant}{article}{
      author={Tessitore, Gianmario},
      author={Zabczyk, Jerzy},
       title={Invariant measures for stochastic heat equations},
        date={1998},
        ISSN={0208-4147},
     journal={Probab. Math. Statist.},
      volume={18},
      number={2, Acta Univ. Wratislav. No. 2111},
       pages={271\ndash 287},
      review={\MR{1671596}},
}

\bib{walsh:86:introduction}{incollection}{
      author={Walsh, John~B.},
       title={An introduction to stochastic partial differential equations},
        date={1986},
   booktitle={\'{E}cole d'\'{e}t\'{e} de probabilit\'{e}s de {S}aint-{F}lour,
  {XIV}---1984},
      series={Lecture Notes in Math.},
      volume={1180},
   publisher={Springer, Berlin},
       pages={265\ndash 439},
         url={https://doi.org/10.1007/BFb0074920},
      review={\MR{876085}},
}

\bib{yang.zhang:14:existence}{article}{
      author={Yang, Juan},
      author={Zhang, Tusheng},
       title={Existence and uniqueness of invariant measures for {SPDE}s with
  two reflecting walls},
        date={2014},
        ISSN={0894-9840},
     journal={J. Theoret. Probab.},
      volume={27},
      number={3},
       pages={863\ndash 877},
         url={https://doi.org/10.1007/s10959-012-0448-x},
      review={\MR{3245987}},
}

\bib{zambotti:01:reflected}{article}{
      author={Zambotti, Lorenzo},
       title={A reflected stochastic heat equation as symmetric dynamics with
  respect to the 3-d {B}essel bridge},
        date={2001},
        ISSN={0022-1236,1096-0783},
     journal={J. Funct. Anal.},
      volume={180},
      number={1},
       pages={195\ndash 209},
         url={https://doi.org/10.1006/jfan.2000.3685},
      review={\MR{1814427}},
}

\bib{zambotti:02:integration*1}{article}{
      author={Zambotti, Lorenzo},
       title={Integration by parts formulae on convex sets of paths and
  applications to {SPDE}s with reflection},
        date={2002},
        ISSN={0178-8051},
     journal={Probab. Theory Related Fields},
      volume={123},
      number={4},
       pages={579\ndash 600},
         url={https://doi.org/10.1007/s004400200203},
      review={\MR{1921014}},
}

\bib{zambotti:04:occupation}{article}{
      author={Zambotti, Lorenzo},
       title={Occupation densities for {SPDE}s with reflection},
        date={2004},
        ISSN={0091-1798,2168-894X},
     journal={Ann. Probab.},
      volume={32},
      number={1A},
       pages={191\ndash 215},
         url={https://doi.org/10.1214/aop/1078415833},
      review={\MR{2040780}},
}

\bib{zambotti:17:random}{book}{
      author={Zambotti, Lorenzo},
       title={Random obstacle problems},
      series={Lecture Notes in Mathematics},
   publisher={Springer, Cham},
        date={2017},
      volume={2181},
        ISBN={978-3-319-52095-7; 978-3-319-52096-4},
         url={https://doi.org/10.1007/978-3-319-52096-4},
        note={Lecture notes from the 45th Probability Summer School held in
  Saint-Flour, 2015},
      review={\MR{3616274}},
}

\bib{zambotti:21:brief}{article}{
      author={Zambotti, Lorenzo},
       title={A brief and personal history of stochastic partial differential
  equations},
        date={2021},
        ISSN={1078-0947,1553-5231},
     journal={Discrete Contin. Dyn. Syst.},
      volume={41},
      number={1},
       pages={471\ndash 487},
         url={https://doi.org/10.3934/dcds.2020264},
      review={\MR{4182330}},
}

\bib{zhang:12:large}{article}{
      author={Zhang, Tusheng},
       title={Large deviations for invariant measures of {SPDE}s with two
  reflecting walls},
        date={2012},
        ISSN={0304-4149},
     journal={Stochastic Process. Appl.},
      volume={122},
      number={10},
       pages={3425\ndash 3444},
         url={https://doi.org/10.1016/j.spa.2012.06.003},
      review={\MR{2956111}},
}

\end{biblist}
\end{bibdiv}
